\documentclass[12pt]{amsart}

\newtheorem{theorem}{Theorem}[section]
\newtheorem{lemma}[theorem]{Lemma}
\newtheorem{proposition}[theorem]{Proposition}
\newtheorem{corollary}[theorem]{Corollary}
\newtheorem{definition}[theorem]{Definition}
\newtheorem{remark}[theorem]{Remark}

\numberwithin{equation}{section}

\begin{document}

\baselineskip=15pt

\title[Vector bundles over a Klein bottle]{Stable real
algebraic vector bundles over a Klein bottle}

\author[U. N. Bhosle]{Usha N. Bhosle}

\address{School of Mathematics, Tata Institute of Fundamental
Research, Homi Bhabha Road, Mumbai 400005, India}

\email{usha@math.tifr.res.in}

\author[I. Biswas]{Indranil Biswas}

\address{School of Mathematics, Tata Institute of Fundamental
Research, Homi Bhabha Road, Mumbai 400005, India}

\email{indranil@math.tifr.res.in}

\subjclass[2000]{Primary 14H60; Secondary 14P99}

\keywords{Stable bundle, real algebraic curve, Klein bottle}

\date{}

\begin{abstract}
Let $X$ be a geometrically connected smooth projective curve of
genus one, defined over the field of real numbers, such that
$X$ does not have any real points. We classify the isomorphism
classes of all stable real algebraic vector bundles over $X$.
\end{abstract}

\maketitle

\section{Introduction}

A \textit{Klein surface} is a
geometrically connected smooth projective curve,
defined over the field of real numbers, satisfying the condition
that it does not have any points defined over $\mathbb R$
\cite{AG}.
For a geometrically connected smooth curve $X$ defined over
$\mathbb R$, the corresponding complex curve $X_{\mathbb C}
\, :=\, X\times_{\mathbb R}{\mathbb C}$ is equipped with
an anti--holomorphic involution $\sigma$ given by the involution
of $\mathbb C$ defined by $z\, \longmapsto\, \overline{z}$.
If $X$ does not have any points defined over $\mathbb R$, then
$\sigma$ does not have any fixed points. Therefore,
in that case the quotient space $X_{\mathbb C}/\sigma$ is
a nonorientable topological surface. If, furthermore, $X$
is projective of genus one, then $X_{\mathbb C}/\sigma$ is
evidently a Klein bottle.

By a Klein bottle we will mean a geometrically connected smooth
projective curve of genus one defined over $\mathbb R$ that does
not have any real points.

In \cite{At2}, Atiyah classified the
isomorphism classes of indecomposable vector
bundles over a smooth elliptic curve defined over $\mathbb C$.
His method can be generalized to classify indecomposable
vector bundles
on smooth elliptic curves defined over $\mathbb R$.
By a smooth elliptic curve we mean a geometrically
connected projective group of dimension one. Therefore,
a smooth elliptic curve defined over $\mathbb R$ has
an obvious point
defined over $\mathbb R$, namely the identity element for the
group law.

A Klein bottle does not have a real point,
consequently it does not admit any
real algebraic vector bundle of odd degree. On the other
hand, when we consider real algebraic vector bundles of
even degree, a Klein bottle has many more
stable vector bundles than a smooth elliptic curve defined
over $\mathbb R$. The space of stable holomorphic vector
bundles of rank $r$ and degree $d$
over a smooth elliptic curve defined over $\mathbb C$
is empty if $\text{gcd}(r\, ,d)\,\not=\, 1$,
and it is of complex dimension one
if $\text{gcd}(r\, ,d)\,=\, 1$. It follows from what we
prove here that given any $r\, >\, 0$ and $d$,
the real dimension of the space
of stable real algebraic vector bundles of rank $2r$ and
degree $2d$ over a Klein bottle is two.

Our aim here is to investigate the real algebraic vector
bundles over a Klein bottle. We classify the isomorphism
classes of real algebraic vector bundles of rank one and
rank two over a Klein bottle. We also
classify the isomorphism classes of
stable real algebraic vector bundles of all ranks.

Take a Klein bottle $X$. For any integers $r'$ and $d'$,
let ${\mathcal M}_{X_{\mathbb C}}(r',d')$ be the moduli
space of stable vector bundles over
$X_{\mathbb C}\, :=\, X\times_{\mathbb R}{\mathbb C}$ of
rank $r'$ and degree $d'$. Let $\text{Pic}^0(X_{\mathbb C})$
be the Picard group parametrizing
holomorphic line bundles over
$X_{\mathbb C}$ of degree zero. The discrete subgroup of
$\text{Pic}^0(X_{\mathbb C})$ defined by the line bundles of
order $r'$ will be denoted by $\Gamma_{r'}$. Let
$\text{Pic}^0(X)\, \subset\, \text{Pic}^0(X_{\mathbb C})$
be the subgroup of real algebraic line bundles over $X$
of degree zero. Set
$$
\Gamma^{\mathbb R}_{r'}\, :=\, \text{Pic}^0(X)\bigcap
\Gamma_{r'}\, .
$$
Let $\sigma\, :\, X_{\mathbb C}\, \longrightarrow\,
X_{\mathbb C}$ be the anti--holomorphic involution of the
complexification $X_{\mathbb C}$.
For a holomorphic vector bundle $F$ over $X_{\mathbb C}$,
let $\overline{F}$ be the smooth complex vector bundle 
whose underlying real vector bundle is the one
underlying the holomorphic vector bundle $F$, while
the complex structure on the fibers of $\overline{F}$
is the conjugate complex structure of the fibers of $F$. 
The pullback $\sigma^*\overline{F}$ has a natural
structure of a holomorphic vector bundle.

Given a real algebraic vector bundle $E$ on $X$, let
$E_{\mathbb C}\, :=\, E\bigotimes_{\mathbb R}{\mathbb C}$
be the corresponding complex algebraic vector bundle over
$X_{\mathbb C}$. We show that if $E$ is stable, then either 
$E_{\mathbb C}$ is stable, or $E_{\mathbb C}$ is isomorphic to
$F\bigoplus \sigma^*\overline{F}$, where $F$ is a stable
vector bundle over $X_{\mathbb C}$.
Since the degree and rank of a stable vector bundle on 
$X_{\mathbb C}$ must be coprime, this implies that 
if $E$ is a stable real algebraic vector bundle over $X$
of rank $r$ and degree $d$, then either $r$ and $d$ are
mutually coprime with $d$ even, or
$\text{gcd}(r\, ,d)\, =\, 2$; see Corollary \ref{cor3}.
If $\text{gcd}(r\, ,d)\, =\, 1$, then $E_{\mathbb C}$ is stable 
and the classification of $E$ in this case is similar
to that of stable holomorphic vector bundles on $X_{\mathbb C}$. 
In case $\text{gcd}(r\, ,d)\, =\, 2$, one has 
$E_{\mathbb C} \cong F\bigoplus \sigma^*\overline{F}, F$ stable.
Thus the classification problem of stable
vector bundles reduces to the problem of
determining
which $F\bigoplus \sigma^*\overline{F}$ are complexifications of 
real vector bundles on $X$, we solve the latter problem.
Our main results may be summed up as follows.
\begin{itemize}
\item Assume that $\text{gcd}(r\, ,d)\, =\, 1$
with $d$ an even integer and
$r\, \geq\, 1$. Then the isomorphism classes
of stable real algebraic vector bundles over $X$
of rank $r$ and degree $d$ are parametrized
by $\text{Pic}^0(X)/\Gamma^{\mathbb R}_{r'}$ (see
Theorem \ref{prop5}).

\item If $\text{gcd}(r\, ,d)\, =\, 2$ with $d\, =\, 2d'$
and $r\, \geq\, 1$,
where $d'$ is an odd integer, then
the isomorphism classes
of stable real algebraic vector bundles over $X$
of rank $r$ and degree $d$ are canonically parametrized
by the quotient space ${\mathcal M}_{X_{\mathbb
C}}(\frac{r}{2},d')/({\mathbb Z}/2{\mathbb Z})$
for the involution
of ${\mathcal M}_{X_{\mathbb C}}(\frac{r}{2},d')$ defined by
$W\, \longrightarrow\, \sigma^*\overline{W}$.
The space of isomorphism classes
of stable real algebraic vector bundles over $X$
of rank $r$ and degree $d$ can be identified
with the quotient space
$({\rm Pic}^0(X_{\mathbb C})/\Gamma_r)/({\mathbb Z}/2{\mathbb Z})$
for the involution of ${\rm Pic}^0(X_{\mathbb C})/\Gamma_r$ defined
by $L\, \longmapsto\, \sigma^*\overline{L}$
(see Proposition \ref{hro} and Corollary \ref{cor4}).

\item Let $\text{gcd}(r\, ,d)\, =\, 2$ with $r\, \geq\, 1$
and $d\, =\, 2d'$, where
$d'$ is an even integer. Then the isomorphism classes
of stable real algebraic vector bundles over $X$
of rank $r$ and degree $d$ are parametrized by the quotient
by ${\mathbb Z}/2{\mathbb Z}$ of the complement
$({\rm Pic}^0(X_{\mathbb C})/\Gamma_r)\setminus
({\rm Pic}^0(X)/\Gamma^{\mathbb R}_r)$ (see Corollary
\ref{cor5} and Proposition \ref{hro2}).
\end{itemize}

To classify the real algebraic vector
bundles of rank two we first show that an indecomposable real
algebraic vector bundle of rank $2$ over $X$ is semistable.
Let $E$ be an indecomposable real algebraic vector bundle
over $X$ of rank two and degree zero. Then exactly one of
the following two statements is valid:
\begin{enumerate}
\item There is a real algebraic line bundle $\xi$
over $X$ of degree zero such that $E$ is obtained as the
unique nontrivial extension of $\xi$ by itself.

\item The vector bundle $E$ is stable.
\end{enumerate}

The space of real algebraic vector bundles
over $X$ of rank two and degree two are classified using
the second statement in the above classification of stable
vector bundles.

We will give a brief description of the sections.
In Section \ref{sec2}
we recall the classification of Klein bottles.
We show that the real algebraic vector bundle on a Klein bottle
is uniquely determined by the corresponding vector bundle
on the complex elliptic curve. In Section \ref{sec3}, we
classify all the real algebraic line bundles on any given Klein
bottle. In Section \ref{sec4}, we show that a semistable 
real algebraic vector bundle over a Klein bottle decomposes
into a direct sum of semistable vector bundles. In Section
\ref{sec5}, we classify all the isomorphism classes of
real algebraic vector bundles of rank two over a Klein
bottle. In Section \ref{sec6}, we investigate the
higher rank vector bundles over a Klein bottle. We classify
all the isomorphism classes of polystable
real algebraic vector bundles of rank at least three
over a Klein bottle. While Section
\ref{sec6} is independent of Section \ref{sec5},
it borrows heavily from the methods in Section \ref{sec5}.
In fact, Section \ref{sec5} can be simplified using
Lemma \ref{st.de.}.

{}From our point of view it
is difficult to parametrize the filtration data of
semistable vector bundles which are not polystable. For this
reason Section \ref{sec6} falls short of classifying all
real algebraic vector bundles over a Klein bottle.

\section{Preliminaries}\label{sec2}

Let $X$ be a geometrically connected
smooth projective curve of genus
one defined over the field $\mathbb R$ of real numbers.
We assume that $X$ does not have any points
defined over $\mathbb R$.
Such a curve is known as a \textit{Klein bottle}
\cite{AG}.

Let $X_{\mathbb C}\, =\, X\times_{\mathbb R}
{\mathbb C}$ be the complex projective curve
obtained by base change to $\mathbb C$. So
$X_{\mathbb C}$ is an irreducible smooth curve
of genus one.

The complex manifold $X_{\mathbb C}$ is equipped
with an anti--holomorphic involution
\begin{equation}\label{sig}
\sigma\, :\, X_{\mathbb C}\, \longrightarrow\,
X_{\mathbb C}\, ,
\end{equation}
which is defined by the conjugation involution
of field $\mathbb C$ that sends any $z$
to $\overline{z}$. Since $X$ does not have any points
defined over $\mathbb R$, this involution $\sigma$ does
not have any fixed points. Conversely, a smooth elliptic
curve defined over $\mathbb C$, equipped with a fixed point
free anti--holomorphic involution, gives a
geometrically connected smooth projective curve of genus
one, defined over $\mathbb R$,
which does not have any points
defined over $\mathbb R$. Therefore, by a Klein bottle
we will also mean a smooth elliptic
curve defined over $\mathbb C$ equipped with a fixed point
free anti--holomorphic involution.

Let $V$ be a real algebraic vector bundle over
the Klein bottle $X$. Let
$$
V_{\mathbb C}\, :=\, V\otimes_{\mathbb R}{\mathbb C}
$$
be the corresponding complex algebraic vector bundle over
$X_{\mathbb C}$. Let $\overline{V_{\mathbb C}}$ be the
$C^\infty$ complex vector bundle over $X_{\mathbb C}$
whose underlying real vector bundle is $V_{\mathbb C}$, and
the complex structure of each fiber of
$\overline{V_{\mathbb C}}$ is the conjugate of the
complex structure of the fibers of
$V_{\mathbb C}$. The vector bundle
$\overline{V_{\mathbb C}}$ does not have a natural
holomorphic structure. However, the vector bundle
$\sigma^*\overline{V_{\mathbb C}}$ has a natural
holomorphic --- hence algebraic --- structure, where
$\sigma$ is the involution in eqn. \eqref{sig}. The
fact that for any holomorphic function $f$ on $\mathbb C$,
the function $z\, \longrightarrow\, \overline{f(\overline{z})}$
is also holomorphic ensures that
$\sigma^*\overline{V_{\mathbb C}}$ has a natural
holomorphic structure.

The involution $\sigma$ lifts to an algebraic isomorphism
\begin{equation}\label{tau}
\delta\, :\, V_{\mathbb C}\, \longrightarrow\,
\sigma^*\overline{V_{\mathbb C}}
\end{equation}
such that the composition
\begin{equation}\label{comp.}
V_{\mathbb C}\, \stackrel{\delta}{\longrightarrow}\,
\sigma^*\overline{V_{\mathbb C}}
\, \stackrel{\sigma^*\overline{\delta}}{\longrightarrow}\,
\sigma^*\overline{\sigma^*\overline{V_{\mathbb C}}}
\, =\, V_{\mathbb C}
\end{equation}
is the identity map of $V_{\mathbb C}$; since 
$\sigma^2 \, =\, \text{Id}_{X_{\mathbb C}}$,
and $\overline{\overline{F}}\, =\, F$ for any
complex vector bundle $F$, it follows that
$\sigma^*\overline{\sigma^*\overline{V_{\mathbb C}}}$
is canonically identified with $ V_{\mathbb C}$.

Let $W$ be another real algebraic vector bundle over $X$.
Set $W_{\mathbb C}\, =\, W\bigotimes_{\mathbb R}{\mathbb C}$.
Let
$$
\delta_W\, :\, W_{\mathbb C}\, \longrightarrow\,
\sigma^*\overline{W_{\mathbb C}}
$$
be the isomorphism as in eqn. \eqref{tau} for $W$. Then the
real vector space $H^0(X,\, \text{Hom}(V\, ,W))$ is identified
with subspace of $H^0(X_{\mathbb C},\,
\text{Hom}(V_{\mathbb C}\, ,W_{\mathbb C}))$ consisting of
all homomorphisms $f$ such that the diagram
$$
\begin{matrix}
V_{\mathbb C} & \stackrel{f}{\longrightarrow} & W_{\mathbb C}\\
\Big\downarrow\delta && \Big\downarrow \delta_W \\
\sigma^*\overline{V_{\mathbb C}} &
\stackrel{\sigma^*\overline{f}}{\longrightarrow}
& \sigma^*\overline{W_{\mathbb C}}
\end{matrix}
$$
is commutative.

The following two categories are equivalent:
\begin{enumerate}
\item the category of real algebraic vector bundles over
$X$ with ${\mathcal O}_X$--linear homomorphisms as morphisms, and

\item the category whose objects are pairs $(V\, ,\delta)$, where
$V$ is a holomorphic (= algebraic)
vector bundle over $X_{\mathbb C}$ and
$$
\delta\, :\, V\, \longrightarrow\,\sigma^*\overline{V}
$$
is an algebraic isomorphism, such
that the composition $(\sigma^*\overline{\delta}) \circ \delta$
is the identity map of $V$ (see eqn. \eqref{comp.}), and
the morphisms from $(V\, ,\delta)$ to $(F\, ,\eta)$ are all
${\mathcal O}_{X_{\mathbb C}}$--linear homomorphisms
$$
\phi\, :\, V\, \longrightarrow\, F
$$
such that the diagram
$$
\begin{matrix}
V & \stackrel{\phi}{\longrightarrow} & F\\
\Big\downarrow\delta && \Big\downarrow \eta\\
\sigma^*\overline{V} &
\stackrel{\sigma^*\overline{\phi}}{\longrightarrow}
& \sigma^*\overline{F}
\end{matrix}
$$
is commutative.
\end{enumerate}

\begin{lemma}\label{lem0}
Let $E_1$ and $E_2$ be two real algebraic vector
bundles over a Klein bottle $X$. Let
$V_i\, :=\, E_i\bigotimes_{\mathbb R}{\mathbb C}$,
$i\, =\,1,2$, be the corresponding holomorphic vector
bundles over $X_{\mathbb C}\, =\, X\times_{\mathbb R}
{\mathbb C}$. If the two holomorphic vector bundles
$V_1$ and $V_2$ are isomorphic, then $E_1$ is isomorphic
to $E_2$.
\end{lemma}

\begin{proof}
Assume that $V_1$ is isomorphic to $V_2$. The
isomorphisms from $V_1$ to $V_2$ constitute a
Zariski open dense subset
\begin{equation}\label{cu}
{\mathcal U}\, \subset\, H^0(X_{\mathbb C},\,
V^*_1\otimes V_2)\, .
\end{equation}

For $i\, =\, 1,2$, let
$$
\delta_i\, :\, V_i\, \longrightarrow\,
\sigma^*\overline{V_i}
$$
be the isomorphism as in eqn. \eqref{tau}.
We have a conjugate linear involution
$$
f\, :\, H^0(X_{\mathbb C},\,
V^*_1\otimes V_2)\, \longrightarrow\,
H^0(X_{\mathbb C},\, V^*_1\otimes V_2)
$$
defined by $T\, \longmapsto\, \delta^{-1}_2
\circ (\sigma^*\overline{T})\circ \delta_1$, where
$T\, \in\, H^0(X_{\mathbb C},\, V^*_1\bigotimes V_2)$.

Since $f$ is a conjugate linear involution,
the subset
$$
{\mathcal S} \, :=\, \{\alpha\, \in\, H^0(X_{\mathbb C},\,
V^*_1\otimes V_2)\, \mid\,f(\alpha)\, =\, \alpha\}\, \subset\,
H^0(X_{\mathbb C},\, V^*_1\otimes V_2)
$$
is a totally real $\mathbb R$--linear subspace.
In other words, we have
$$
H^0(X_{\mathbb C},\, V^*_1\otimes V_2)\,
=\, {\mathcal S}\bigoplus \sqrt{-1}{\mathcal S}\, .
$$
This implies that ${\mathcal S}$ is Zariski
dense in $H^0(X_{\mathbb C},\, V^*_1\bigotimes V_2)$.

Now it is easy to see that any
$$
T\, \in\, {\mathcal S}\bigcap{\mathcal U}\, ,
$$
where $\mathcal U$ is the Zariski open dense subset
in eqn. \eqref{cu}, gives an isomorphism of $E_1$
with $E_2$. This completes the proof of the lemma.
\end{proof}

We will now recall the classification of the Klein bottles.

Take any real number $\tau$, with $\tau\, >\, 0$.
Let
\begin{equation}\label{la.}
\Lambda\, :=\, \{m\tau\sqrt{-1}+n\, \in\, {\mathbb C}\,\mid\,
m,n\, \in\, {\mathbb Z}\}
\end{equation}
be the lattice in $\mathbb C$ generated by $1$ and $\tau$.
Let
\begin{equation}\label{wt}
\widetilde{\sigma}\, :\, {\mathbb C}\, \longrightarrow\,
{\mathbb C}
\end{equation}
be the map defined by $z\, \longmapsto\, \overline{z} +
1/2$.

Let
\begin{equation}\label{yt}
Y_\tau\, :=\, {\mathbb C}/\Lambda
\end{equation}
be the quotient space, which is an elliptic curve. The map
$\widetilde{\sigma}$ in eqn. \eqref{wt} descends to a self--map
\begin{equation}\label{si}
\sigma\, :\, Y_\tau\, \longrightarrow\, Y_\tau
\end{equation}
of the quotient space. It is easy to see that
$\sigma$ is a fixed--point free
anti--holomorphic involution of $Y_\tau$.

Therefore, the pair $(Y_\tau\, ,\sigma)$ is a Klein bottle.

See \cite[p. 64, Theorem 1.9.8]{AG} for a proof of the
following theorem.

\begin{theorem}\label{thm1}
The isomorphism classes of Klein bottles, equivalently,
the isomorphism classes of
geometrically connected smooth projective curves of
genus one, defined over $\mathbb R$, without any real points,
are as follows: Given any Klein bottle $Y$, there is a
unique real number $\tau\, >\, 0$, such that the Klein bottle
$(Y_\tau\, ,\sigma)$ (see eqn. \eqref{yt} and eqn. \eqref{si})
is isomorphic to $Y$.
\end{theorem}

\begin{remark}
{\rm In \cite[p. 64, Theorem 1.9.8]{AG}, Klein bottles
are classified as follows. For each elliptic curve
$Y_\tau$ defined in eqn. \eqref{yt},
consider the fixed--point free anti--holomorphic
involution
$$
\sigma'\, :\, Y_\tau\, \longrightarrow\, Y_\tau
$$
induced by the self-map
$$
z\, \longmapsto\, -\overline{z} +\frac{\sqrt{-1}\tau}{2}
$$
of $\mathbb C$. If $\tau\, >\,1$, then
$(Y_\tau\, ,\sigma)$ and $(Y_\tau\, ,\sigma')$ are not
isomorphic. For $\tau\, =\, 1$, the two Klein bottles
$(Y_1\, ,\sigma)$ and $(Y_1\, ,\sigma')$ are isomorphic
by the isomorphism induced by the self-map of
$\mathbb C$ defined by $z\, \longmapsto\, \sqrt{-1}z$.
In \cite[Theorem 1.9.8]{AG} it is proved that
any Klein bottle is isomorphic to exactly one
Klein bottle from the union of the following three sets
of Klein bottles:
\begin{enumerate}
\item $(Y_\tau\, ,\sigma)$, with $\tau\, >\, 1$,

\item $(Y_\tau\, ,\sigma')$, with $\tau\, >\, 1$, and

\item $(Y_1\, ,\sigma)$.
\end{enumerate}
We note that $(Y_\tau\, ,\sigma)$ is isomorphic to
$(Y_{1/\tau}\, ,\sigma')$. To see this consider
the isomorphism
$$
f_\tau\, :\, Y_\tau\, \longrightarrow\, Y_{1/\tau}
$$
induced by the self-map of $\mathbb C$ defined by
$$
z\, \longmapsto\, \frac{\sqrt{-1}z}{\tau}\, .
$$
it is easy to see that $\sigma\, =\, f^{-1}_\tau\circ
\sigma'\circ f_\tau$. Therefore, the two Klein bottles
$(Y_1\, ,\sigma)$ and $(Y_1\, ,\sigma')$ are isomorphic.
Therefore, Theorem \ref{thm1} follows from Theorem 1.9.8
of \cite{AG}.}
\end{remark}

In the next section we will describe the
real algebraic line bundles over a Klein bottle.

\section{Line bundles of the Klein bottle}\label{sec3}

Let $(Y\, , \sigma)$ be a Klein bottle. For any
integer $d$, let $\text{Pic}^d(Y)$
denote the Picard variety parametrizing isomorphism classes
of holomorphic line bundles of degree $d$ over the compact
connected Riemann surface $Y$. Let $Y_{\mathbb R}$ be the
smooth projective curve defined over $\mathbb R$ corresponding
to $(Y\, , \sigma)$. Therefore, $Y\, =\, Y_{\mathbb R}
\times_{\mathbb R} {\mathbb C}$.
The Picard variety $\text{Pic}^d(Y)$ is the complexification
of a variety defined over $\mathbb R$. To explain this, we
note that the complex manifold $\text{Pic}^d(Y)$
is equipped with an anti--holomorphic involution
\begin{equation}\label{e3}
\sigma_d\, :\, \text{Pic}^d(Y)\, \, \longrightarrow\,
\text{Pic}^d(Y)
\end{equation}
defined by $L\, \longmapsto\, \sigma^*\overline{L}\, \in\,
\text{Pic}^d(Y)$.
Here $\overline{L}$ is the smooth complex line bundle over
$Y$ whose underlying real vector bundle of rank two is $L$,
and the complex structure on the fibers of $\overline{L}$
is the conjugate complex structure of the fibers of $L$
(see Section \ref{sec2}). Since $\sigma$ is an
anti--holomorphic map, the pull--back $\sigma^*\overline{L}$
has a natural holomorphic structure. The holomorphic structure
of $\sigma^*\overline{L}$ is determined by the following
condition. A smooth section
of $\sigma^*\overline{L}$ over an analytic open subset
$U\, \subset\, Y$ is holomorphic if the corresponding
smooth section of $L$ over $\sigma(U)$ is holomorphic.
Since $\sigma_d$ is an anti--holomorphic involution of
$\text{Pic}^d(Y)$, the pair $(\text{Pic}^d(Y)\, ,\sigma_d)$
define a geometrically irreducible smooth projective
variety defined over the field of real numbers. Consider
the conjugate complex structure on the
real manifold underlying the complex manifold
$\text{Pic}^d(Y)$ (if $J$ is the almost complex structure on
$\text{Pic}^d(Y)$, then the conjugate complex complex
structure is $-J$). The complex manifold defined by this
conjugate complex structure is the twist of the complex
variety $\text{Pic}^d(Y)$ by the nontrivial element in the
Galois group $\text{Gal}(({\mathbb C}/
{\mathbb R})\, =\, {\mathbb Z}/2{\mathbb Z}$.
Consequently, the anti--holomorphic
involution $\sigma_d$, which is actually an
involutive holomorphic
isomorphism between $\text{Pic}^d(Y)$ and
its Galois conjugate, defines a projective
variety defined over $\mathbb R$.
Hence the real points of the real algebraic variety
$(\text{Pic}^d(Y)\, ,\sigma_d)$ are parametrized by
the fixed points of the involution $\sigma_d$.
The complex variety given by the real variety
$(\text{Pic}^d(Y)\, ,\sigma_d)$ using the inclusion of
$\mathbb R$ in $\mathbb C$ is identified with
$\text{Pic}^d(Y)$. (See \cite[Ch. I, \S~1]{Si}
and \cite[Ch. I, \S~4]{Si} for more details.)

Let $\text{Pic}^d(Y_{\mathbb R})$ denote the
space of all real algebraic line bundles
of degree $d$ over the real curve $Y_{\mathbb R}$
corresponding to $(Y\, , \sigma)$. So
$\text{Pic}^0(Y_{\mathbb R})$ is a group, and
$\text{Pic}^d(Y_{\mathbb R})$ is an affine space
for it, which means that $\text{Pic}^0(Y_{\mathbb R})$
acts freely transitively on $\text{Pic}^d(Y_{\mathbb R})$.
We note that this does not mean that
$\text{Pic}^d(Y_{\mathbb R})$ is nonempty. If $\xi$ is
a real algebraic line bundle over the curve $Y_{\mathbb R}$
corresponding to $(Y\, , \sigma)$, then we have
$$
\xi_{\mathbb C}\, \cong\, \sigma^*\overline{\xi_{\mathbb C}}\, ,
$$
where $\xi_{\mathbb C}\,=\,\xi\bigotimes_{\mathbb R}{\mathbb C}$
is the algebraic line bundle over $Y$ given by $\xi$.
In particular, for any
real algebraic line bundle $\xi$ over $Y_{\mathbb R}$ of degree
$d$, the point in $\text{Pic}^d(Y)$ corresponding to the
line bundle $\xi\bigotimes_{\mathbb R}{\mathbb C}$ is defined
over $\mathbb R$.

However, the converse is not true in general. A point of
$\text{Pic}^d(Y)$ defined over $\mathbb R$ need not
correspond to a real algebraic line bundle over $Y_{\mathbb R}$.
In fact, an algebraic
line bundle $L\, \in\, \text{Pic}^d(Y)$ corresponds to a
real algebraic line bundle over the curve $Y_{\mathbb R}$ if
and only if there is a holomorphic isomorphism
\begin{equation}\label{e4}
\beta\, :\, L\, \longrightarrow\, \sigma^*\overline{L}
\end{equation}
such that the composition homomorphism
\begin{equation}\label{e5}
L\, \stackrel{\beta}{\longrightarrow}\, \sigma^*\overline{L}
\, \stackrel{\sigma^*\overline{\beta}}{\longrightarrow}\, 
\sigma^*\overline{\sigma^*\overline{L}}\, =\, L
\end{equation}
is the identity automorphism of $L$.

\begin{proposition}\label{prop1}
For any integer $n$, the pair $({\rm Pic}^d(Y)\, ,\sigma_d)$
is isomorphic to the pair $({\rm Pic}^{d+2n}(Y)\, ,\sigma_{d+2n})$.

For any integer $d$, the involution $\sigma_{2d+1}$ of
${\rm Pic}^{2d+1}(Y)$ does not have any fixed points.

For any integer $d$, there is a real algebraic line bundle
of degree $2d$ over the curve $Y_{\mathbb R}$.
\end{proposition}

\begin{proof}
Take any point $x_0\, \in\, Y$. Let $D\, =\, x_0+\sigma (x_0)$ be
the divisor. The algebraic line bundle ${\mathcal O}_Y(D)$ over
$Y$ defined by $D$ corresponds to a real algebraic line bundle
over the curve $Y_{\mathbb R}$. Indeed,
$$
\sigma^*\overline{{\mathcal O}_Y(D)}\, =\,
{\mathcal O}_Y(\sigma(D))\, =\, {\mathcal O}_Y(D)\, .
$$
Furthermore the tautological isomorphism 
$\beta\, :\, {\mathcal O}_Y(D)\, \longrightarrow\,
\sigma^*\overline{{\mathcal O}_Y(D)}$ satisfies the
condition that $(\sigma^*\overline{\beta})\circ\beta
\,=\, \text{Id}_{{\mathcal O}_Y(D)}$.

Therefore, the morphism
$$
{\rm Pic}^d(Y)\,\longrightarrow\, {\rm Pic}^{d+2n}(Y)
$$
defined by $L\, \longmapsto\, L\bigotimes {\mathcal O}_Y(nD)$
intertwines the involutions $\sigma_{d}$ and $\sigma_{d+2n}$.
Thus, the two pairs $({\rm Pic}^d(Y)\, ,\sigma_d)$ and
$({\rm Pic}^{d+2n}(Y)\, ,\sigma_{d+2n})$ are isomorphic.

To prove the second assertion, note that
$({\rm Pic}^1(Y)\, ,\sigma_1)$ is canonically identified
with $(Y\, ,\sigma)$ by sending any point $y\in\, Y$
to the line bundle ${\mathcal O}_Y(y)$. Since the involution
$\sigma$ of $Y$ does not have any fixed points, the second
assertion follows immediately.

Since the trivial holomorphic line bundle ${\mathcal O}_Y$
equipped with the smooth involution defined by
$f\, \longmapsto\, \overline{f\circ\sigma}$, where $f$
is any locally defined holomorphic function on $Y$, satisfies
the condition that the composition
homomorphism as in eqn. \eqref{e5}
is the identity automorphism of ${\mathcal O}_Y$, the third
assertion follows using the first assertion. Therefore,
the proof of the proposition is complete.
\end{proof}

To investigate the real algebraic variety defined by
the degree zero line bundles over a Klein bottle, we will
use the explicit description of a Klein bottle noted
in Theorem \ref{thm1}.

Take any positive real number $\tau$. Let
$(Y_\tau\, ,\sigma)$ be the corresponding Klein bottle
(see eqn. \eqref{yt} and eqn. \eqref{si}). For any
$z\, \in\, {\mathbb C}$, the image of $z$ in the
quotient $Y_\tau$ (see eqn. \eqref{yt}) will be denoted
by $\underline{z}$. We will identify
${\rm Pic}^0(Y_\tau)$ with $Y_\tau$ using the holomorphic map
\begin{equation}\label{phi}
\phi\, :\, Y_\tau\, \longrightarrow\, {\rm Pic}^0(Y_\tau)
\end{equation}
defined by $x\, \longmapsto\, {\mathcal O}_{Y_\tau}(
{x}-\underline{0})$.
We note that $\phi$ does not intertwine the involution
$\sigma$ of $Y_\tau$
and the involution $\sigma_0$ of ${\rm Pic}^0(Y_\tau)$ defined
in eqn. \eqref{e3}. Indeed, the involution $\sigma$ does
not have any fixed points, while $\sigma_0$ has fixed points,
as shown in Proposition \ref{prop1}. We also note that
$\phi$ is an isomorphism of algebraic groups.

The smooth projective curve defined over $\mathbb R$ corresponding
to $(Y_\tau\, ,\sigma)$ will be denoted by $Y^{\mathbb R}_\tau$.

With the above notation, we have the following theorem:

\begin{theorem}\label{thm2}
Let ${\rm Pic}^0(Y_\tau)^{\sigma_0}\, \subset\,
{\rm Pic}^0(Y_\tau)$ be the fixed point set for
the involution $\sigma_0$ of ${\rm Pic}^0(Y_\tau)$.
\begin{enumerate}

\item{}
The fixed point set ${\rm Pic}^0(Y_\tau)^{\sigma_0}$
is the disjoint union 
\begin{equation}\label{se.}
{\rm Pic}^0(Y_\tau)^{\sigma_0}\, =\, \{\phi(\underline{r}) \, \mid
\, r\, \in\, {\mathbb R}; 0\leq r <1\}\bigcup \{\phi (\underline{r
+\sqrt{-1}\tau/2})\, \mid\, r\, \in\, {\mathbb R}; 0\leq r <1\}\, .
\end{equation}

\item{} The subset $\{\phi(\underline{r})\, \mid\, r\,
\in\, {\mathbb R};
0\leq r <1\}\, \subset\, {\rm Pic}^0(Y_\tau)^{\sigma_0}$
consists of all real algebraic line bundles over the
curve $Y^{\mathbb R}_\tau$.

\item{} For any holomorphic line bundle
$$
L\, \in\, \{\phi(\underline{r+\sqrt{-1}\tau/2})\, \mid\, r\,
\in\, {\mathbb R}; 0\leq r <1\}\, \subset\,
{\rm Pic}^0(Y_\tau)^{\sigma_0}\, ,
$$
there is no real algebraic line bundle $\xi$ over the
curve $Y^{\mathbb R}_\tau$ such that $L\, =\, \xi
\bigotimes_{\mathbb R}{\mathbb C}$.
\end{enumerate}
\end{theorem}

\begin{proof}
{}From the definition of $\phi$ we have
\begin{equation}\label{e7}
\phi(\underline{z})\, =\,
{\mathcal O}_{Y_\tau}(\underline{z}-\underline{0})\, .
\end{equation}
Using this, from the definition of $\sigma$ in eqn. \eqref{si}
we have
\begin{equation}\label{e8}
\sigma^*\overline{\phi(\underline{z})}\, =\,
{\mathcal O}_{Y_\tau}(\underline{\widetilde{\sigma}(z)}
-\underline{\widetilde{\sigma}(0)})
\, =\, {\mathcal O}_{Y_\tau}(\underline{\overline{z}+1/2}
-\underline{1/2})\,=\, \phi(\underline{\overline{z}+1/2})\otimes
\phi(\underline{1/2})^*
\end{equation}
for all $z\, \in\, {\mathbb C}$,
where $\widetilde{\sigma}$ is defined in eqn. \eqref{wt}.
Since the holomorphic map
$$
{\mathbb C}\, \longrightarrow\, \text{Pic}^0(Y_\tau)
$$
defined by $z \, \longmapsto\, \phi(\underline{z})$
is a group homomorphism, from eqn. \eqref{e8} it follows
that
\begin{equation}\label{e82}
\sigma^*\overline{\phi(\underline{z})}\,=\,
\phi(\underline{\overline{z}})\, .
\end{equation}

Now fix any $z\, \in\, {\mathbb C}$ such that the following
three conditions hold:
\begin{itemize}
\item $0\,\leq \,\text{Re}(z)\, <\, 1$,

\item $0\, \leq\, \text{Im}(z)\, <\,\tau$, and

\item for the algebraic group isomorphism $\phi$
defined in eqn. \eqref{phi},
\begin{equation}\label{e6}
\phi(\underline{z})\, \cong\, \sigma^*
\overline{\phi(\underline{z})}\, ,
\end{equation}
\end{itemize}
where $\sigma$ is defined in eqn. \eqref{si}, and
$\text{Re}(z)$ (respectively, $\text{Im}(z)$)
is the real (respectively, imaginary) part of $z$.

We recall that a line bundle
$L\, \in\, {\rm Pic}^0(Y_\tau)$ is
a fixed point of the involution $\sigma_0$ if and only
if $L\, \cong\, \sigma^*\overline{L}$.
Now, from eqn. \eqref{e82} and eqn. \eqref{e6} we have
\begin{equation}\label{e9}
z-\overline{z}\, \in\, \Lambda\, ,
\end{equation}
where $\Lambda$ is the lattice in \eqref{la.}.

As $0\, \leq\, \text{Im}(z)\, <\,\tau$, from eqn.
\eqref{e9} we conclude that the equality in
eqn. \eqref{se.} is valid. This proves the first
statement in the theorem.

To prove the second statement,
let $I$ denote the closed interval $[0\, ,1/2]\, \subset\,
{\mathbb R}$. Consider the map
$$
\widetilde{\gamma}\, :\,
{\mathbb C}\times I\, \longrightarrow\, {\mathbb C}
$$
defined by $\widetilde{\gamma} (z\, ,t) \, =\, z+t$. This
map $\widetilde{\gamma}$ induces a map
\begin{equation}\label{g1}
\gamma\, :\, Y_\tau\times I\, \longrightarrow\, Y_\tau\, .
\end{equation}
For any $t\, \in\, I$, let
\begin{equation}\label{g2}
\gamma_t \, :\, Y_\tau\, \longrightarrow\, Y_\tau
\end{equation}
be the map defined by $z\, \longmapsto\, \gamma ((z,t))$.

A flat connection on a holomorphic line bundle is said
to be \textit{compatible} with the holomorphic structure
if every locally defined flat section is holomorphic.
If $L$ is a holomorphic line bundle over $Y_\tau$ equipped
with a flat connection compatible
with the holomorphic structure, then parallel translations
along $\gamma$ (defined in eqn. \eqref{g1})
gives a holomorphic isomorphism of the
holomorphic line bundle $L$ with $\gamma_{1/2}^*L$.

Take any $z_0\, \in\, {\mathbb C}$. We will give an explicit
description of the holomorphic line bundle
$\phi(\underline{z_0})$ over $Y_\tau$.

Consider the holomorphically trivial line bundle
$$
\xi_0\, =\, Y_\tau\times {\mathbb C}
$$
over $Y_\tau$ with fiber $\mathbb C$. In other words,
the sheaf of holomorphic sections of $\xi_0$ is the
sheaf of holomorphic functions on $Y_\tau$. Therefore,
the Dolbeault operator on $\xi_0$ defining its holomorphic
structure is simply the differential that sends any locally
defined complex valued smooth function $f$ to
$\overline{\partial}f$.
This line bundle $\xi_0$ is equipped with the
Hermitian structure defined by $\vert(y\, ,c)\vert^2
\, =\, c\overline{c}$, where $y\, \in\, Y_\tau$
and $c\, \in\, {\mathbb C}$.

Let
\begin{equation}\label{om.}
\omega\, :=\, \overline{\partial}z
\end{equation}
be the $(0\, ,1)$--form on $Y_\tau\, =\,
{\mathbb C}/\Lambda$ (here $z$ is the
natural coordinate on the covering surface
${\mathbb C}$ of $Y_\tau$). Note that the
$(0\, ,1)$--form on ${\mathbb C}$ descends to the quotient
$Y_\tau$. The holomorphic line
bundle $\phi(\underline{z_0})$ is the smooth line bundle $\xi_0$
equipped with the Dolbeault operator
\begin{equation}\label{dz0}
\overline{\partial}_{z_0}\, :=\, \overline{\partial}-
\frac{{\pi}z_0\omega}{\tau}\, ,
\end{equation}
where $\omega$ is the form in eqn. \eqref{om.}. This expression
for $\overline{\partial}_{z_0}$ can be deduced from the
combination of the fact that the image of the line bundle
$\phi(\underline{z_0})$ in
$$
\text{Pic}^0(Y_\tau) \, =\, \frac{H^0(Y_\tau,\,
\Omega_{Y_\tau})^*}{H_1(Y_\tau,\, {\mathbb Z})}
$$
is given by integration of holomorphic one--forms on
$Y_\tau$ along any smooth path connecting $\underline{z_0}$
to $\underline{0}$ and the fact that
$$
\int_{Y_\tau} {\partial} z\wedge \overline{\partial} \overline{z}
\, =\, -2\tau\sqrt{-1}\, .
$$

Consider the connection
\begin{equation}\label{co}
D_{z_0}\, :=\,
\partial+\frac{{\pi}\overline{z_0}\overline{\omega}}{\tau}
+\overline{\partial}_{z_0}
\end{equation}
on the line bundle $\phi(\underline{z_0})$, where
$\overline{\partial}_{z_0}$ is the Dolbeault operator
in eqn. \eqref{dz0}, and $\partial$ is the $(1\, ,0)$--part
of the de Rham differential. It is straight--forward to check
that $D_{z_0}$ (defined in eqn. \eqref{co}) is a flat
connection on $\phi(\underline{z_0})$ compatible with
the holomorphic structure.

In fact $D_{z_0}$
is the unique unitary flat connection on the holomorphic
line bundle $\phi(\underline{z_0})$, and a flat
unitary metric on $\phi(\underline{z_0})$ is given by the
constant metric on the smooth trivial line bundle
$Y_\tau\times_{\mathbb R} {\mathbb C}$ (recall that
$\phi(\underline{z_0})$ was constructed by putting
the Dolbeault operator $\overline{\partial}_{z_0}$ on the
trivial line bundle). However, we will not need this for
the proof.

Continuing with the proof of the theorem, for any
$z_0\, \in \, {\mathbb C}$ we have
$$
\gamma_{1/2}^*\phi(\underline{z_0})\, =\,
\phi(\underline{z_0})\, ,
$$
where $\gamma_{1/2}$ is defined in eqn. \eqref{g2};
note that
$\phi(\underline{z_0})\, =\, {\mathcal O}_{Y_{\tau}}
(\underline{z_0} -\underline{0})$, and 
$\gamma_{1/2}^*\phi(\underline{z_0})\, =\,
{\mathcal O}_{Y_{\tau}}
(\underline{z_0-1/2} -\underline{-1/2})$. 

Now fix any $z_0\, \in \, {\mathbb R}$.
As $\overline{z_0}\, =\, z$, using
eqn. \eqref{e82} it follows that the holomorphic line bundle
$\sigma^*\overline{\phi(\underline{z_0})}$ is
isomorphic to $\phi(\underline{z_0})$. Therefore,
the holomorphic line bundle $\gamma_{1/2}^*\phi(\underline{z_0})$
is identified with the holomorphic line bundle
$\sigma^*\overline{\phi(\underline{z_0})}$.
The parallel translations along $\gamma$
(defined in eqn. \eqref{g1}) for the connection
$D_{z_0}$ (defined in eqn. \eqref{co}) gives a holomorphic
isomorphism of $\phi(\underline{z_0})$
with $\sigma^*\overline{\phi(\underline{z_0})}$.
Let
$$
\eta\, :\, \phi(\underline{z_0})\, \longrightarrow\,
\sigma^*\overline{\phi(\underline{z_0})}
$$
be the isomorphism obtained this way.

The composition
$$
(\sigma^*\overline{\eta})\circ \eta\,:\,
\phi(\underline{z_0})\, \longrightarrow\,
\sigma^*\overline{\sigma^*\overline{\phi(\underline{z_0})}}
\, =\, \phi(\underline{z_0})
$$ 
is the holonomy of the connection 
$D_{z_0}$ (defined in eqn. \eqref{co}) along the
closed path
\begin{equation}\label{path}
S^1\, \longrightarrow\, Y_\tau
\end{equation}
defined by $\exp(2\pi\sqrt{-1}t) \, \longmapsto\, \underline{t}$,
where $0\, \leq\, t\, < \,1$. It is straight--forward
to check that $D_{z_0}$ has trivial holonomy along this
loop.

Consequently, the composition
$(\sigma^*\overline{\eta})\circ \eta$ coincides with the
identity map of $\phi(\underline{z_0})$. Therefore,
$\phi(\underline{z_0})$ corresponds to a real algebraic
line bundle over the real curve $Y^{\mathbb R}_\tau$.
This proves the second statement in the theorem.

To prove the third statement in the theorem,
take any $z_0\, =r +\sqrt{-1}\tau/2$, where $r\, \in\, {\mathbb R}$.
Again we have
$$
\gamma_{1/2}^*\phi(\underline{z_0})\, \cong\,
\phi(\underline{z_0})\, ,
$$
and
the holomorphic line bundle $\gamma_{1/2}^*\phi(\underline{z_0})$
is identified with
$\sigma^*\overline{\phi(\underline{z_0})}$.
As before, let
$$
\eta\, :\, \phi(\underline{z_0})\, \longrightarrow\,
\sigma^*\overline{\phi(\underline{z_0})}
$$
be the isomorphism given by parallel translation,
for the connection $D_{z_0}$, along $\gamma$.
It is straight--forward to check that $D_{z_0}$ has
holonomy $-1$ along the closed loop
defined in eqn.\eqref{path}. Therefore, the composition
$$
(\sigma^*\overline{\eta})\circ \eta\,:\,
\phi(\underline{z_0})\, \longrightarrow\,
\sigma^*\overline{\sigma^*\overline{\phi(\underline{z_0})}}
\, =\, \phi(\underline{z_0})
$$
satisfies the identity
\begin{equation}\label{id.r}
(\sigma^*\overline{\eta})\circ \eta\,=\, -
\text{Id}_{\phi(\underline{z_0})}\, .
\end{equation}

If we replace the isomorphism $\eta$ by
$\eta'\, :=\, c\eta$, where $c\, \in\, {\mathbb C}\setminus\{0\}$,
then using eqn. \eqref{id.r},
$$
(\sigma^*\overline{\eta'})\circ\eta'\, =\, c\overline{c} \cdot
\sigma^*\overline{\eta}\circ\eta\, =\, -c\overline{c}\, .
$$
Therefore, there is no isomorphism
$$
\eta'\, :\, \phi(\underline{z_0})\, \longrightarrow\,
\sigma^*\overline{\phi(\underline{z_0})}
$$
such that $(\sigma^*\overline{\eta'})\circ \eta'\, =\,1$.
This completes the proof of the theorem.
\end{proof}

Combining Proposition \ref{prop1} and Theorem \ref{thm2},
we obtain a complete classification of real algebraic line
bundles on a Klein bottle. In particular, we have the
following corollary.

\begin{corollary}\label{cor-2}
Let $X$ be a Klein bottle.
\begin{itemize}
\item The degree of any real algebraic line bundle over
$X$ is an even integer.

\item The space of real algebraic line bundles over $X$ of
degree zero is parametrized by ${\mathbb R}/{\mathbb Z}$.
\end{itemize}
\end{corollary}

\section{Semistable vector bundles and
indecomposability}\label{sec4}

Let $X$ be a Klein bottle, that is, a geometrically connected
smooth projective curve of genus one, defined over $\mathbb R$,
which does not have any points defined over $\mathbb R$. A real
algebraic vector bundle $E$ over $X$ is called \textit{stable}
(respectively, \textit{semistable}) if for all nonzero proper
subbundles $F\, \subset\, E$, the inequality
$$
\frac{\text{degree}(F)}{\text{rank}(F)}\, <\,
\frac{\text{degree}(E)}{\text{rank}(E)}
$$
(respectively, $\frac{\text{degree}(F)}{\text{rank}(F)}\,
\leq\, \frac{\text{degree}(E)}{\text{rank}(E)}$) holds;
see \cite{Ne} for more details. A semistable vector bundle
over $X$ is called \textit{polystable} if it is a direct sum
of stable vector bundles.

\begin{lemma}\label{lem1}
Let $X_{\mathbb C}\, =\,
X\times_{\mathbb R} {\mathbb C}$ be the
corresponding complex elliptic curve.
Let $E$ be a real algebraic vector bundle over $X$. Let
$E_{\mathbb C}\, =\, E\bigotimes_{\mathbb R} {\mathbb C}$
be the corresponding vector bundle over $X_{\mathbb C}$.
\begin{enumerate}
\item The vector bundle $E$ is semistable if and only
if the vector bundle $E_{\mathbb C}$ over $X_{\mathbb C}$
is semistable.

\item The vector bundle $E$ over $X$
splits into a direct sum of semistable vector bundles.

\item The vector bundle $E$ is polystable if and only
if $E_{\mathbb C}$ is polystable.
\end{enumerate}
\end{lemma}

\begin{proof}
If $E_{\mathbb C}$ is semistable, then clearly
$E$ is semistable. To prove the converse, assume that
$E_{\mathbb C}$ is not semistable. Let
\begin{equation}\label{hn}
0\, =\, V_0\,\subset\, V_1\,\subset\, V_2 \,\subset\,\cdots 
\, \subset\, V_{\ell-1} \, \subset\, V_\ell \, =\,E_{\mathbb C}
\end{equation}
be the Harder--Narasimhan filtration of $E_{\mathbb C}$.

Let $\sigma$ be the anti--holomorphic involution
of $X_{\mathbb C}$ as in eqn. \eqref{sig}. We have an
isomorphism 
$$
\delta\, :\, E_{\mathbb C}\, \longrightarrow\,
\sigma^*\overline{E_{\mathbb C}}
$$
as in eqn. \eqref{tau}. Since both the
filtrations
$$
0\, =\,\delta(V_0)\, \subset\,
\delta(V_1)\,\subset\,\delta(V_2) \,\subset\,\cdots\, \subset\,
\delta(V_{\ell-1})\,\subset\,
\delta(V_\ell)\,=\,\delta(E_{\mathbb C})
\, =\, \sigma^*\overline{E_{\mathbb C}}
$$
and
$$
0\, =\, \sigma^*\overline{V_0} \, \subset\,
\sigma^*\overline{V_1}\,\subset\, \sigma^*\overline{V_2} \,\subset
\,\cdots\, \subset\, \sigma^*\overline{V_{\ell-1}} \, \subset\,
\sigma^*\overline{V_\ell} \, =\,\sigma^*\overline{E_{\mathbb C}}
$$
of $\sigma^*\overline{E_{\mathbb C}}$ satisfy the conditions of
the Harder--Narasimhan filtration, from the uniqueness of the
Harder--Narasimhan filtration we conclude that
$\delta(V_i)\, =\, \sigma^*\overline{V_i}$ for all
$i\, \in\, [1\, ,\ell]$. Therefore, the subbundle $V_i\, \subset
\, E_{\mathbb C}$ is given by a subbundle of $E$ by base
change from $\mathbb R$ to $\mathbb C$.
Hence the vector bundle $E$ is not semistable.

To prove the second statement in the lemma, we first
note that if $V$ and $W$ are two semistable vector bundles
over $X_{\mathbb C}$ with
$$
\frac{\text{degree}(V)}{\text{rank}(V)}\, >\,
\frac{\text{degree}(W)}{\text{rank}(W)}\, ,
$$
then $H^1(X_{\mathbb C},\, W^*\bigotimes V)\, =\, 0$. Indeed,
by Serre duality,
$$
H^1(X_{\mathbb C},\, W^*\otimes V)^*\, =\,
H^0(X_{\mathbb C},\, W^*\otimes V\otimes K_{X_{\mathbb C}})\,
\cong \, H^0(X_{\mathbb C},\, W^*\otimes V)\, =\, 0
$$
as $K_{X_{\mathbb C}}$ is trivializable.
Therefore, if $E$ over $X$ is not semistable, then the
filtration in eqn. \eqref{hn} splits completely. In other
words, there is a filtration preserving isomorphism
$$
E_{\mathbb C}\, \longrightarrow\, \bigoplus_{i=1}^\ell
\frac{V_i}{V_{i-1}}\, .
$$
We already noted that each $V_i/V_{i-1}$ corresponds to
a real algebraic vector bundle over $X$. Therefore,
the second statement follows from Lemma \ref{lem0}.

To prove the final statement in the lemma, take a
polystable vector bundle $E$ over $X$. So, in particular,
$E$ is semistable, and hence the corresponding vector bundle
$E_{\mathbb C}$ over $X_{\mathbb C}$ is semistable. Assume
that $E_{\mathbb C}$ is not polystable. Let
$$
W\, \subset\, E_{\mathbb C}
$$
be the socle of $E_{\mathbb C}$. In other words, $W$ is
the unique maximal polystable subbundle of $E_{\mathbb C}$
with
$$
\frac{\text{degree}(W)}{\text{rank}(W)}\, =\,
\frac{\text{degree}(E_{\mathbb C})}{\text{rank}(E_{\mathbb C})}
$$
(see \cite[p. 23, Lemma 1.5.5]{HL}).

Clearly, both the
subbundles $\delta(W)$ and $\sigma^*\overline{W}$
of $\sigma^*\overline{(E_{\mathbb C})}$ satisfy the conditions
for a socle. From the uniqueness of the socle we conclude that
that the two subbundles $\delta(W)$ and $\sigma^*\overline{W}$
coincide. Therefore, $W$ corresponds to a subbundle of $E$.
Let $W_{\mathbb R}\, \subset\, E$ be the subbundle such that
$W\, =\, W_{\mathbb R}\bigotimes_{\mathbb R}{\mathbb C}$.

Since $E$ is polystable, if $W_{\mathbb R}\, \not=\, E$,
then there is another nonzero subbundle $W'\, \subset\, E$
such that
$$
W_{\mathbb R}\oplus W' \, =\, E\, .
$$
In that case, the socle of
$$
E_{\mathbb C}\, =\,
(W_{\mathbb R}\otimes_{\mathbb R}{\mathbb C})\oplus
(W'\otimes_{\mathbb R}{\mathbb C})\, =\,
W\oplus (W'\otimes_{\mathbb R}{\mathbb C})
$$
is $W\bigoplus W''$, where $W''$ is the socle of
$W'\bigotimes_{\mathbb R}{\mathbb C}$. But this
contradicts the condition that $W$ is the
socle of $E_{\mathbb C}$. Thus, $W_{\mathbb R}\,
=\, E$. This completes the proof of the lemma.
\end{proof}

The statement (2) in Lemma \ref{lem1} has the following
corollary:

\begin{definition}
{\rm A real algebraic vector bundle over $X$ is called}
decomposable {\rm if it splits into a direct sum of
real algebraic vector bundles of positive ranks. A
real algebraic vector bundle is called} indecomposable 
{\rm if it is not decomposable.}
\end{definition}

\begin{corollary}\label{cor0}
Any indecomposable real algebraic vector bundle over $X$ 
is semistable.
\end{corollary}

\section{Vector bundles of rank two over a Klein bottle}\label{sec5}

Let $X$ be a geometrically connected smooth projective curve of
genus one, defined over $\mathbb R$, which does not have any points
defined over $\mathbb R$. In this section we will describe the
isomorphism classes of all real algebraic vector bundles of rank
two over $X$.

We first recall that Proposition \ref{prop1} says
that for any even integer $n$, there is a real algebraic line
bundle over $X$ of degree $n$. Also, in Theorem \ref{thm2}
and Proposition \ref{prop1}, all line bundles over $X$ are
described. Therefore, it is enough to classify the
isomorphism classes of indecomposable
real algebraic vector bundles over $X$ of rank two and degree
$d$, with $0\, \leq\, d\, \leq\, 3$.

For any vector bundle $E$ over $X$ of rank two and degree $d$,
we have
$$
d\, =\, \text{degree}(E)\, =\, \text{degree}(\bigwedge^2 E)\, .
$$
Therefore, from Proposition \ref{prop1} it follows that $d$
is an even integer. Thus, we need to consider the two
cases of $d\, =\, 0$ and $d\, =\, 2$.

Let $X_{\mathbb C}\, :=\, X\times_{\mathbb R}{\mathbb C}$
be the smooth elliptic curve defined over $\mathbb C$. Let
\begin{equation}\label{si.c}
\sigma\, :\, X_{\mathbb C}\, \longrightarrow\,
X_{\mathbb C}
\end{equation}
be the fixed point free anti--holomorphic involution
as in eqn. \eqref{sig} given by the conjugation of $\mathbb C$.

Take any complex algebraic line bundle $L$ over $X_{\mathbb C}$.
Consider the algebraic vector bundle
\begin{equation}\label{2si}
S(L)\, :=\, L\oplus \sigma^*\overline{L}
\end{equation}
of rank two over $X_{\mathbb C}$. For $S(L)$ we have
$$
\sigma^*\overline{S(L)}\, =\, \sigma^*\overline{L\oplus
\sigma^*\overline{L}}\, =\, \sigma^*\overline{L}\oplus L\, .
$$
Therefore, there is a canonical isomorphism
\begin{equation}\label{2sl}
\sigma_L\, :\, S(L)\, \longrightarrow\, \sigma^*\overline{S(L)}
\end{equation}
defined by $(v_1\, ,v_2)\, \longmapsto\, (v_2\, ,v_1)$.

It is easy to check that the
composition $(\sigma^*\overline{\sigma_L})\circ\sigma_L$
is the identity automorphism of $S(L)$. Therefore, the
pair $(S(L)\, , \sigma_L)$ give a real algebraic vector
bundle over $X$ of rank two (see Section \ref{sec2}).
Let $V(L)$ denote the real algebraic vector
bundle over $X$ of rank two given by $(S(L)\, , \sigma_L)$.

\begin{proposition}\label{prop2}
Let $L$ and $L'$ be two complex algebraic line bundles over
$X_{\mathbb C}$. Let $V(L)$ and $V(L')$ be the corresponding
real algebraic
vector bundles of rank two over $X$ (see the above construction).
The two vector bundles $V(L)$ and $V(L')$ are isomorphic if and
only if either $L$ is isomorphic to $L'$ or
$\sigma^*\overline{L}$ is isomorphic to $L'$.

If ${\rm degree}(L)\, =\, 1$, then the
above vector bundle $V(L)$ is stable.

Let $E$ be an indecomposable real algebraic vector bundle
over $X$ of rank two and degree two. Then there is a
complex algebraic line bundle $L$ of degree one over 
$X_{\mathbb C}$ such that the corresponding vector bundle
$V(L)$ over $X$ is isomorphic to $E$.
\end{proposition}

\begin{proof}
First assume that $V(L)\, \cong\, V(L')$. So we have
$$
L\oplus \sigma^*\overline{L}\, \cong\, L'\oplus
\sigma^*\overline{L'}\, .
$$
Now from \cite[p. 315, Theorem 2]{At1} it follows immediately
that either $L\, \cong\,L'$ or $\sigma^*\overline{L}\, \cong\,L'$.

If $L\, \cong\,L'$ or $\sigma^*\overline{L}\, \cong\,L'$, then
$S(L)\, \cong\, S(L')$, where $S(L)$ is defined in
eqn. \eqref{2si}. If $S(L)\, \cong\, S(L')$, from Lemma
\ref{lem0} it follows that $V(L)$ and $V(L')$ are isomorphic.
This completes the proof of the first part.

Now assume that ${\rm degree}(L)\, =\, 1$.
Since $L\bigoplus \sigma^*\overline{L}$ is polystable, the
vector bundle $V(L)$ is polystable (see Lemma \ref{lem1}(3)).
As there are no line bundles over $X$ of degree one
(see Proposition \ref{prop1}), the polystable vector bundle $V(L)$
of rank two and degree two must be stable.

We will now prove the final part of the proposition.
Since $E$ is indecomposable, from Corollary \ref{cor0}
we know that $E$ is semistable.

We will show that $E$ is stable. To prove this, assume that
$E$ is not stable. Therefore, there is a line subbundle
$$
\xi\, \subset\, E
$$
with $\text{degree}(\xi)\, =\,1$. But this contradicts the
second assertion in Proposition \ref{prop1}. Therefore, we
conclude that the vector bundle $E$ is stable.

Since $E$ is stable, from Lemma \ref{lem1}(3) it follows
immediately that the corresponding vector bundle
\begin{equation}\label{cf}
E_{\mathbb C}\, :=\, E\otimes_{\mathbb R} {\mathbb C}
\end{equation}
over $X_{\mathbb C}$ is polystable. There are no stable
vector bundles over $X_{\mathbb C}$ of rank two and
degree two \cite[p. 428, Lemma 11]{At2}; see also Theorem
\ref{thm3} given in Section \ref{st62}. Consequently,
$$
E_{\mathbb C}\, =\, L_1\oplus L_2\, ,
$$
with $\text{degree}(L_1)\, =\, 1\, =\, \text{degree}(L_2)$.

Let
$$
\delta\, :\, L_1\oplus L_2\, =\, E_{\mathbb C}\,
\longrightarrow\, \sigma^*\overline{E_{\mathbb C}}\, =\,
\sigma^*\overline{L_1}\oplus \sigma^*\overline{L_2}
$$
be the isomorphism corresponding to $E$ (see eqn. \eqref{tau}).
Now from \cite[p. 315, Theorem 2]{At1} it follows
that either $L_1\, \cong\,\sigma^*\overline{L_1}$ or
$L_2\, \cong\,\sigma^*\overline{L_1}$. On the other hand, since
$\text{degree}(L_1)\, =\, 1$, from the second statement in
Proposition \ref{prop1} it follows that
$L_1$ is not isomorphic to $\sigma^*\overline{L_1}$. Thus,
we conclude that
$$
L_2\, \cong\,\sigma^*\overline{L_1}\, .
$$

Therefore, $E_{\mathbb C}\, \cong\, L_1\bigoplus
\sigma^*\overline{L_1} \,=\, S(L_1)$, where
$S(L_1)$ is defined in eqn. \eqref{2si}. Now from Lemma
\ref{lem0} it follows that the real algebraic vector bundle
$E$ over $X$ is isomorphic to $V(L_1)$. This completes the
proof of the proposition.
\end{proof}

Proposition \ref{prop2} has the following corollary:

\begin{corollary}\label{cor1}
The isomorphism classes of
indecomposable real algebraic vector bundles
over $X$ of rank two and degree two are parametrized by
the quotient space
$$
X_{\mathbb C}/\sigma\, =\,
{\rm Pic}^1(X_{\mathbb C})/\sigma_1\, ,
$$
where $\sigma_1$ is the fixed point free involution defined
by $L\, \longmapsto\, \sigma^*\overline{L}$.

Any indecomposable real algebraic vector bundle
over $X$ of rank two and degree two is stable.
\end{corollary}

We will now investigate the indecomposable vector 
bundles over $X$ of rank two and degree zero.

Since the canonical line bundle of $X_{\mathbb C}$ is
trivializable, from Lemma \ref{lem0} it follows that
the canonical line bundle $K_X$ of $X$ is also trivializable.

Let $\xi$ be a real algebraic line bundle over $X$. Since
$K_X$ is trivializable, from Serre duality we have
$$
\dim H^1(X,\, \xi^*\otimes\xi)\, =\, \dim H^0(X,\, {\mathcal O}_X)
\, =\, 1\, .
$$
Therefore, for $\xi$, there is a unique, up to an isomorphism,
real algebraic vector bundle $W(\xi)$ over $X$ which is a nontrivial
extension of $\xi$ by $\xi$. In other words, there is a unique,
up to an isomorphism, indecomposable
real algebraic vector bundle $W(\xi)$ over $X$
that fits in a short exact sequence
\begin{equation}\label{eqw}
0\, \longrightarrow\, \xi\, \longrightarrow\, W(\xi)
\, \longrightarrow\, \xi\, \longrightarrow\, 0\, .
\end{equation}

\begin{proposition}\label{prop3}
Let $E$ be a semistable real algebraic vector bundle
over $X$ of rank two and degree zero. Then exactly one of
the following three statements is valid.
\begin{enumerate}
\item There is a real algebraic line bundle $\xi$
over $X$ of degree zero such that $W(\xi)$ is isomorphic to $E$
(the vector bundle $W(\xi)$ is defined in eqn. \eqref{eqw}).

\item There are real algebraic line bundles $\xi_1$ and $\xi_2$
over $X$ of degree zero such that $\xi_1\bigoplus\xi_2$ is
isomorphic to $E$.

\item There is a line bundle $L\, \in\, {\rm Pic}^0(X_{\mathbb C})
\setminus {\rm Pic}^0(X)$ such that $V(L)$ is isomorphic to
$E$, where ${\rm Pic}^0(X)$ denotes the group of line bundles
of the form $L\bigotimes_{\mathbb R}{\mathbb C}$ with $L$ being
a real algebraic line bundle over $X$ of degree zero,
and $V(L)$ is defined prior to Proposition \ref{prop2}.
\end{enumerate}
For any $L\, \in\, {\rm Pic}^0(X_{\mathbb C})
\setminus {\rm Pic}^0(X)$, the real algebraic vector bundle
$V(L)$ over $X$ is stable.
\end{proposition}

\begin{proof}
We will show that the above three cases
correspond to following three mutually exclusive
and exhaustive cases:
\begin{enumerate}
\item{} $E$ is not polystable;

\item $E$ is polystable, but not stable; and

\item $E$ is stable.
\end{enumerate}

The semistable
vector bundle $E$ is not polystable if and only if
it fits in a non--split short exact sequence
$$
0\, \longrightarrow\, \xi\, \longrightarrow\, E
\, \longrightarrow\, \eta\, \longrightarrow\, 0
$$
of real algebraic vector bundles over $X$, where
$\text{degree}(\xi)\, =\, 0\, =\, \text{degree}(\eta)$.
If $\eta\, \not=\, \xi$, then
$$
H^1(X,\, \eta^*\otimes \xi)\, =\, 0
$$
and the above exact sequence splits. In other words,
$\eta\, =\, \xi$ if $E$ is not polystable. Therefore,
we conclude that $E\, \cong\, W(\xi)$ for some real
algebraic line bundle $\xi$ over $X$ of degree zero if
and only if $E$ is not polystable.

It is obvious that the semistable vector
bundle $E$ is polystable but not stable if and only if
$E\, \cong\, \xi_1\bigoplus \xi_2$, where $\xi_1$ and $\xi_2$ are
some real algebraic line bundles of degree zero.

Take any line bundle $L\, \in\, {\rm Pic}^0(X_{\mathbb C})$. We will
show that $V(L)$ is stable if and only if
$L\, \in\, {\rm Pic}^0(X_{\mathbb C})\setminus {\rm Pic}^0(X)$.

Since the vector bundle $S(L)$ constructed in eqn. \eqref{2si}
is polystable, from Lemma \ref{lem1}(3) it follows that $V(L)$
is polystable. Now, $V(L)$ is not stable if and only if it
decomposes as a direct sum of two line bundles. If $V(L)$
decomposes as a direct sum of two line bundles, from
\cite[p. 315, Theorem 2]{At1} it follows
that $L\, \in\, {\rm Pic}^0(X)$. Conversely, if
$L\, \in\, {\rm Pic}^0(X)$, then $S(L)\, =\, L\bigoplus L$,
as $L\, =\, \sigma^*\overline{L}$.
Therefore, using Lemma \ref{lem0}
we conclude that for $L\, \in\, {\rm Pic}^0(X_{\mathbb C})$,
the vector bundle $V(L)$ is isomorphic to a direct sum of two
real algebraic line bundles over $X$ of degree zero.

Thus, the vector bundle $V(L)$ is stable if and only if
$L\, \notin\, {\rm Pic}^0(X)$.

Let $E$ be a stable vector bundle over $X$ of rank two and degree
zero. Then its complexification $E_{\mathbb C}$ (defined in
eqn. \eqref{cf}) is polystable (see Lemma \ref{lem1}(3)).
In other words,
$$
E_{\mathbb C}\, =\, L_1\oplus L_2\, ,
$$
where $L_i$, $i\, =\,1,2$, are complex algebraic line bundles
of degree zero over $X_{\mathbb C}$. Let
$$
\delta\, :\, L_1\oplus L_2\, =\, E_{\mathbb C}\,
\longrightarrow\, \sigma^*\overline{E_{\mathbb C}}\, =\,
\sigma^*\overline{L_1}\oplus \sigma^*\overline{L_2}
$$
be the isomorphism as in eqn. \eqref{tau} corresponding to $E$.
Since $E$ is stable, the image $\delta(L_1)$ is not contained
in the subbundle $\sigma^*\overline{L_1}\, \subset\,
\sigma^*\overline{E_{\mathbb C}}$. Consequently, $\delta$ induces a
nonzero homomorphism from $L_1$ to
$\sigma^*\overline{L_2}$. Since any nonzero homomorphism 
between two line bundles of degree zero is an isomorphism,
we have $L_1\, \cong\, \sigma^*\overline{L_2}$.

Since $L_1\, \cong\, \sigma^*\overline{L_2}$, we have
$S(L_2) \, \cong\, E_{\mathbb C}$, where $S(L_2)$
is defined in eqn. \eqref{2si}. Now from Lemma \ref{lem0}
it follows that $E\, \cong\, V(L_2)$.

We have already shown that if $L_2\, \in\, {\rm Pic}^0(X)$,
then $ V(L_2)$ is not stable. This completes the
proof of the proposition.
\end{proof}

The following proposition complements Proposition \ref{prop3}.

\begin{proposition}\label{prop4}
Let $\xi_1$ and $\xi_2$ be real algebraic line bundles
over $X$ of degree zero such that $W(\xi_1)$ is isomorphic
$W(\xi_2)$. Then $\xi_1$ is isomorphic to $\xi_2$.

Let $\xi_i$, $i\, =\, 1,2$, and $\xi'_i$, $i\, =\, 1,2$,
be real algebraic line bundles
over $X$ of degree zero such that
$$
\xi_1\oplus \xi_2\, \cong\, \xi'_1\oplus \xi'_2\, .
$$
Then either $\xi_1\, \cong\, \xi'_1$ or $\xi_1\, \cong\, \xi'_2$.

If $L_1, L_2\, \in\, {\rm Pic}^0(X_{\mathbb C})\setminus
{\rm Pic}^0(X)$ are line bundles such that $V(L_1)$ is
isomorphic to $V(L_2)$, then either $L_1\, \cong\, L_2$
or $L_1\, \cong\, \sigma^*\overline{L_2}$.
\end{proposition}

\begin{proof}
Let
$$
f\, :\, W(\xi_1)\, \longrightarrow\, W(\xi_2)
$$
be an isomorphism. Consider the composition homomorphism
$$
\xi_1 \, \hookrightarrow\, W(\xi_1)\, \stackrel{f}{\longrightarrow}
\, W(\xi_2)\, \longrightarrow\,\xi_2
$$
(see eqn. \eqref{eqw}). Since any nonzero homomorphism
$\xi_1\, \longrightarrow\,\xi_2$ is an isomorphism, and $W(\xi_2)$
is a nontrivial extension of $\xi_2$ by $\xi_2$, the above
composition homomorphism must vanish identically. Hence
$f$ induces an isomorphism of $\xi_1$ with $\xi_2$.

The second statement in the proposition follows from
\cite[p. 315, Theorem 2]{At1} and Lemma \ref{lem0}.

The third statement in the proposition follows from
\cite[p. 315, Theorem 2]{At1}. This completes the proof
of the proposition.
\end{proof}

Corollary \ref{cor0} and Proposition \ref{prop3} together
have the following corollary:

\begin{corollary}\label{cor-1}
Let $E$ be an indecomposable real algebraic vector bundle
over $X$ of rank two and degree zero. Then exactly one of
the following two statements is valid.
\begin{enumerate}
\item The vector bundle $E$ is not polystable.
There is a real algebraic line bundle $\xi$
over $X$ of degree zero such that $W(\xi)$ is isomorphic to $E$
(the vector bundle $W(\xi)$ is defined in eqn. \eqref{eqw}).

\item The vector bundle $E$ is stable.
There is a line bundle $L\, \in\, {\rm Pic}^0(X_{\mathbb C})
\setminus {\rm Pic}^0(X)$ such that $V(L)$ is isomorphic to
$E$, where ${\rm Pic}^0(X)$ denotes the group of line bundles
of the form $L\bigotimes_{\mathbb R}{\mathbb C}$ with $L$ being
a real algebraic line bundle over $X$ of degree zero,
and $V(L)$ is defined prior to Proposition \ref{prop2}.
\end{enumerate}
\end{corollary}

\begin{corollary}\label{cor2}
The isomorphism classes of
stable real algebraic vector bundles
over $X$ of rank two and degree zero are parametrized by
the quotient space $({\rm Pic}^0(X_{\mathbb C})\setminus
{\rm Pic}^0(X))/\sigma_0$, where $\sigma_0$
is the involution defined by
$L\, \longmapsto\, \sigma^*\overline{L}$.

The isomorphism classes of polystable, but not stable, real
algebraic vector bundles over $X$ of rank two and degree zero
are parametrized by the quotient space $({\rm Pic}^0(X)\times
{\rm Pic}^0(X))/({\mathbb Z}/2{\mathbb Z})$ for the action
of ${\mathbb Z} /2{\mathbb Z}$ that switches the factors.

The isomorphism classes of
semistable, but not polystable, real algebraic vector bundles
over $X$ of rank two and degree zero are parametrized by
${\rm Pic}^0(X)$.
\end{corollary}

\begin{remark}
{\rm Corollary \ref{cor2}
describes the isomorphism classes of indecomposable real
algebraic vector bundles over $X$ of rank two and degree zero.
Corollary \ref{cor1} describes the isomorphism classes of
indecomposable real algebraic vector bundles over $X$ of rank
two and degree two. As it was noted earlier, this gives
a classification of all real algebraic vector bundles over
$X$ of rank two.}
\end{remark}

\section{Stable vector bundles of higher rank}\label{sec6}

In this section we will classify all stable real algebraic
vector bundles over a Klein bottle.
We will first investigate the case where the degree
and the rank are mutually coprime.

\subsection{Indecomposable bundles in coprime case}

We begin by recalling a theorem from \cite{At2}.

Fix a positive integer $r$. Take any integer $d$ such that
$r$ and $d$ are mutually coprime. Let $Z$ be a smooth
elliptic curve defined over $\mathbb C$. Let
\begin{equation}\label{gar}
\Gamma_r\, \subset\, \text{Pic}^0(Z)
\end{equation}
be the subgroup of line bundles $L$ over $Z$ such that
$L^{\otimes r}\, \cong\, {\mathcal O}_Z$.

Since $r$ and $d$ are mutually coprime, from
Lemma \ref{lem1}(2) it follows
that an algebraic vector bundle $E$ over $Z$ of rank $r$
and degree $d$ is indecomposable if and only if $E$ is
semistable (see the proof of Lemma \ref{lem1}(2)),
and it also follows that
$E$ is semistable if and only if it is stable.
The following theorem is a special case of
\cite[p. 442, Theorem 10]{At2}.

\begin{theorem}[\cite{At2}]\label{thm.at2}
Let ${\mathcal M}_Z(r,d)$ denote the moduli space
of stable vector bundles over $Z$ of rank $r$ and degree $d$,
where $r$ and $d$ are mutually coprime.
Then ${\mathcal M}_Z(r,d)$ is nonempty.
Take any stable vector bundle $E\, \in\, {\mathcal M}_Z(r,d)$,
and consider the morphism
$$
\gamma_E\, :\, {\rm Pic}^0(Z)\, \longrightarrow\,
{\mathcal M}_Z(r,d)
$$
defined by $L\, \longmapsto\, E\bigotimes L$. Then $\gamma_E$
is a surjective \'etale Galois covering. Furthermore,
$$
\gamma^{-1}_E(E)\, =\, \Gamma_r\, ,
$$
where $\Gamma_r$ is defined in eqn. \eqref{gar}. Therefore,
$\Gamma_r$ is the Galois group for the covering $\gamma_E$.
\end{theorem}

Let $(Y_\tau\, ,\sigma)$ be a Klein bottle defined as in
eqn. \eqref{yt} and eqn. \eqref{si}. Let
$X$ denote the geometrically connected smooth projective curve
of genus one defined over $\mathbb R$ without any real points
corresponding to $(Y_\tau\, ,\sigma)$. Let $d$ be any 
even integer. As before, $r$ is a positive integer such that
$r$ and $d$ are mutually coprime. We note that $r$ is an
odd integer.

We will construct a stable real algebraic vector bundle over
$X$ of rank $r$ and degree $d$. For that purpose, set
$$
\Lambda'\, :=\, \{m\tau\sqrt{-1}+nr\, \in\, {\mathbb C}\,\mid\,
m,n\, \in\, {\mathbb Z}\}\, .
$$
So $\Lambda'$ is a sub-lattice of $\Lambda$ defined in
eqn. \eqref{la.}. Therefore, we have a Galois covering map
\begin{equation}\label{ypt}
f\, :\, Y'_\tau\, :=\, {\mathbb C}/\Lambda'\, \longrightarrow
\, {\mathbb C}/\Lambda \, =:\, Y_\tau
\end{equation}
with Galois group ${\mathbb Z}/r{\mathbb Z}$.
Consider the anti--holomorphic involution
\begin{equation}\label{sip}
\sigma'\, :\, Y'_\tau\, \longrightarrow\, Y'_\tau
\end{equation}
of the quotient space $Y'_\tau$ in eqn. \eqref{ypt}
induced by the map ${\mathbb C}\, \longrightarrow\,
{\mathbb C}$ defined by
$$
z\, \longmapsto\, \overline{z} +\frac{r}{2}\, .
$$
Since $r$ is an odd integer, we have
\begin{equation}\label{ci}
f\circ \sigma'\, =\, \sigma\circ f\, ,
\end{equation}
where $f$ is defined in eqn. \eqref{ypt} and $\sigma$ is
defined in eqn. \eqref{si}.

As $r$ is an odd integer, the involution $\sigma'$
in eqn. \eqref{sip} does not have any fixed points.
Therefore, $(Y'_\tau\, ,\sigma')$ is a Klein bottle.

Fix a pair $(\xi\, ,\beta')$, where $\xi$ is a
holomorphic line bundle over $Y'_\tau$ of degree $d$, and
\begin{equation}\label{bp}
\beta'\, :\, \xi\, \longrightarrow\, (\sigma')^*
\overline{\xi}
\end{equation}
is an isomorphism as in eqn. \eqref{e4}
such that the composition homomorphism
$$
\xi\, \stackrel{\beta'}{\longrightarrow}\, (\sigma')^*
\overline{\xi}\,
\stackrel{(\sigma')^*\overline{\beta'}}{\longrightarrow}\,
(\sigma')^*\overline{(\sigma')^*\overline{\xi}}\, =\, \xi
$$
is the identity automorphism of $\xi$; the map
$\sigma'$ is defined in eqn. \eqref{sip}. Since $d$
is an even integer, such a pair $(\xi\, ,\beta')$ exist;
see Proposition \ref{prop1}. Let $G$
denote the Galois
group for the covering map $f$ in eqn. \eqref{ypt}. Set
\begin{equation}\label{wxi}
\widehat{\xi}\, =\, \bigoplus_{\alpha\, \in\, G} \alpha^*\xi
\end{equation}
which is a holomorphic vector bundle over $Y'_\tau$ of rank $r$
and degree $rd$.

The direct image $f_*\widehat{\xi}$ over $Y_\tau$ is equipped
with an action of $G$. Let
\begin{equation}\label{wx}
W\, :=\, (f_*\widehat{\xi})^G
\end{equation}
be the invariant part of this action of $G$ on $f_*\widehat{\xi}$.
It is easy to see that $W$ is a holomorphic vector bundle over
$Y_\tau$ of rank $r$. Furthermore, the pullback $f^*W$
is canonically identified with $\widehat{\xi}$. This immediately
implies that the degree of $W$ is $d$.

Since the vector bundle $\widehat{\xi}$ in eqn. \eqref{wxi}
is polystable, and $f^*W\, =\, \widehat{\xi}$, we conclude that
the vector bundle $W$ defined in eqn. \eqref{wx}
is also polystable.

Note that for any $\alpha\, \in\, G$, we have
$\alpha\circ\sigma'\, =\, \sigma'\circ\alpha$. Consequently,
$$
(\sigma')^*\overline{\widehat{\xi}}\, =\,
\bigoplus_{\alpha\in G} \alpha^*(\sigma')^*\overline{\xi}\, .
$$
Therefore,
the isomorphism $\beta'$ in eqn. \eqref{bp} gives an isomorphism
\begin{equation}\label{bi}
\widehat{\beta'}\, :=\,
\bigoplus_{\alpha\, \in\, G} \alpha^*\beta'
\, :\, \widehat{\xi}\, =\,
\bigoplus_{\alpha\, \in\, G} \alpha^*\xi\,
\longrightarrow\, \bigoplus_{\alpha\, \in\, G}
 \alpha^*(\sigma')^*\overline{\xi}
\, =\, (\sigma')^*\overline{\widehat{\xi}}\, .
\end{equation}
Since $(\sigma')^*\overline{\beta'}\circ\beta'\, =\, \text{Id}_\xi$,
we have
\begin{equation}\label{wbi} 
((\sigma')^*\overline{\widehat{\beta'}})\circ\widehat{\beta'}\, =\,
\text{Id}_{\widehat{\xi}}\, .
\end{equation}

Using eqn. \eqref{ci}, the isomorphism $\widehat{\beta'}$
in eqn. \eqref{bi} induces an isomorphism
$$
\widehat{\beta}\, :\, f_*\widehat{\xi}\, \longrightarrow\,
\sigma^*(\overline{f_*\widehat{\xi}})\, .
$$
which commutes with the action of the group $G\, =\,
\text{Gal}(f)$ on $f_*\widehat{\xi}$. Therefore,
$\widehat{\beta}$ induces an isomorphism
\begin{equation}\label{deb2}
\beta\, :\, W \, \longrightarrow\,\sigma^*\overline{W}\, ,
\end{equation}
where $W$ is defined in eqn. \eqref{wx}. From eqn. \eqref{wbi}
it follows that $(\sigma^*\overline{\widehat{\beta}})\circ
\widehat{\beta}\, =\, \text{Id}_{f_*\widehat{\xi}}$,
which in turn implies that
\begin{equation}\label{de22}
(\sigma^*\overline{\beta})\circ\beta \, =\, \text{Id}_{W}\, .
\end{equation}

Therefore, the pair $(W\, ,\beta)$ defines a real algebraic
vector bundle over $X$ of rank $r$ and degree $d$,
where $X$ is the real algebraic curve given by
$(Y_\tau\, ,\sigma)$. Let
$W_{\mathbb R}$ denote this real algebraic vector bundle over $X$.
We noted earlier that $W$ is polystable. Therefore, from Lemma
\ref{lem1}(3) it follows that $W_{\mathbb R}$ is
polystable. Since $r$ and $d$ are mutually coprime, we
conclude that the polystable vector bundle $W_{\mathbb R}$ is stable.

We summarize the above construction in the following lemma:

\begin{lemma}\label{lem3}
Let $X$ be a geometrically connected smooth projective curve
of genus one, defined over $\mathbb R$, without any real points.
Let $d$ be an even integer and $r$ a positive integer such
that $r$ and $d$ are mutually coprime. Then there is a stable
real algebraic vector bundle over $X$ of rank $r$ and degree $d$.
\end{lemma}

Let $\text{Pic}^0(X)$ denote the group of real algebraic line
bundles over $X$ of degree zero. From Theorem \ref{thm2} we know
that $\text{Pic}^0(X)$ is identified with
${\mathbb R}/{\mathbb Z}$. The identification sends any
$t\, \in\, {\mathbb R}$ to the line bundle corresponding to
${\mathcal O}_{Y_\tau}(\underline{t}-
\underline{0})$ (as before, $\underline{z}$ is the image
of $z\, \in\, {\mathbb C}$ in the quotient space
$Y_\tau\, =\, {\mathbb C}/\Lambda$). Let 
\begin{equation}\label{garr}
\Gamma^{\mathbb R}_r\, \subset\, \text{Pic}^0(X)
\end{equation}
be the subgroup of line bundles $L$ over $X$ whose order
is a divisor of $r$, i.e.,
$L^{\otimes r}\, \cong\, {\mathcal O}_X$.

\begin{theorem}\label{prop5}
Let $X$ be a geometrically connected smooth projective curve
of genus one, defined over $\mathbb R$, without any real points.
Let $d$ be an even integer and $r$ a positive integer such
that $r$ and $d$ are mutually coprime. Let
${\mathcal M}_X(r,d)$ denote the set of all isomorphism 
classes of stable real algebraic vector bundles over $X$ of rank
$r$ and degree $d$. Fix $E\, \in\, {\mathcal M}_X(r,d)$
(it is nonempty by Lemma \ref{lem3}). Then
the map
$$
\gamma\, :\, {\rm Pic}^0(X)\, \longrightarrow\,
{\mathcal M}_X(r,d)
$$
defined by $L\, \longmapsto\, E\bigotimes L$ is surjective.
Furthermore, $\gamma$ induces a bijection of
${\mathcal M}_X(r,d)$ with the quotient
${\rm Pic}^0(X)/\Gamma^{\mathbb R}_r$, where the
subgroup $\Gamma^{\mathbb R}_r$ is defined in eqn. \eqref{garr}.
\end{theorem}

\begin{proof}
Since $r$ and $d$ are mutually coprime,
if $F\, \in\, {\mathcal M}_X(r,d)$, then the corresponding
vector bundle $F\bigotimes_{\mathbb R}{\mathbb C}$ over
$Y_\tau$ is stable; see Lemma \ref{lem1}(3). Therefore,
the set ${\mathcal M}_X(r,d)$ is in bijective correspondence
with the isomorphism classes of pairs of the form $(W\, ,\beta)$,
where $W$ is a stable complex algebraic vector bundle over
$Y_\tau$ of rank $r$ and degree $d$, and
$$
\beta\, :\, W \, \longrightarrow\,\sigma^*\overline{W}
$$
is an isomorphism as in eqn. \eqref{deb2} such that
$(\sigma^*\overline{\beta})\circ\beta \, =\, \text{Id}_{W}$.

Let ${\mathcal M}_{Y_\tau}(r,d)$ denote the moduli space of
stable vector bundles over $Y_\tau$ of rank $r$ and degree $d$.
Since $r$ and $d$ are mutually coprime, the smooth variety
${\mathcal M}_{Y_\tau}(r,d)$ is projective. The variety
${\mathcal M}_{Y_\tau}(r,d)$ has an anti--holomorphic involution
\begin{equation}\label{is}
\delta_{(r,d)}\, :\, {\mathcal M}_{Y_\tau}(r,d)\, \longrightarrow
\, {\mathcal M}_{Y_\tau}(r,d)
\end{equation}
defined by $V\, \longmapsto\, \sigma^*\overline{V}$. We note
that ${\mathcal M}_X(r,d)$ is
a subset of the fixed--point
set of this involution $\delta_{(r,d)}$. As seen in Theorem
\ref{thm2}, this subset of the fixed--point set
is proper in general.

Let
\begin{equation}\label{1}
\delta_{(1,0)}\, :\, \text{Pic}^0(Y_\tau)\, \longrightarrow
\, \text{Pic}^0(Y_\tau)
\end{equation}
be the anti--holomorphic involution defined by
$L\, \longmapsto\, \sigma^*\overline{L}$.

Fix any stable real algebraic vector bundle $E$ over $X$ of rank
$r$ and degree $d$. Let $(W\, ,\beta)$ be the pair corresponding
to $E$, where $W$ is a complex algebraic vector bundle over
$Y_\tau$ and $\beta$ is an isomorphism as in eqn. \eqref{deb2}.
Therefore, $W$ is isomorphic $\sigma^*\overline{W}$. Let
\begin{equation}\label{gE}
\gamma_W\, :\, {\rm Pic}^0(Y_\tau)\, \longrightarrow\,
{\mathcal M}_{Y_\tau}(r,d)
\end{equation}
be the morphism defined by $L\, \longmapsto\, W\bigotimes
L$. From Theorem \ref{thm.at2} we know that $\gamma_W$
is a surjective
\'etale Galois covering with Galois group $\Gamma_r$, where
$$
\Gamma_r\, \subset\, \text{Pic}(Y_\tau)
$$
is the subgroup defined by all line bundles $L$ such that
$L^{\otimes r}\, \cong\, {\mathcal O}_{Y_\tau}$. Let
\begin{equation}\label{gE2}
\widetilde{\gamma_W}\, :\, {\rm Pic}^0(Y_\tau)/\Gamma_r
\, \longrightarrow\, {\mathcal M}_{Y_\tau}(r,d)
\end{equation}
be the isomorphism given by $\gamma_W$ defined in
eqn. \eqref{gE} (see Theorem \ref{thm.at2}).

Since the subgroup $\Gamma_r\, \subset\, {\rm Pic}^0(Y_\tau)$
is preserved by the anti--holomorphic involution $\delta_{(1,0)}$
defined in eqn. \eqref{1}, the involution $\delta_{(1,0)}$
descends to an anti--holomorphic involution
\begin{equation}\label{de}
\delta\, :\, {\rm Pic}^0(Y_\tau)/\Gamma_r
\, \longrightarrow\, {\rm Pic}^0(Y_\tau)/\Gamma_r
\end{equation}
of the quotient space.

As $W\, \cong\, \sigma^*\overline{W}$, we have
$$
\delta_{(r,d)} \circ\gamma_W\, =\, \gamma_W\circ\delta_{(1,0)}\, ,
$$
where $\delta_{(r,d)}$ (respectively, $\delta_{(1,0)}$)
is defined in eqn. \eqref{is} (respectively, \eqref{1}), and
$\gamma_W$ is defined in eqn. \eqref{gE}. This implies that
\begin{equation}\label{ii}
\delta_{(r,d)}\circ\widetilde{\gamma_W}\, =\,
\widetilde{\gamma_W}\circ\delta\, ,
\end{equation}
where $\widetilde{\gamma_W}$ is constructed in eqn.
\eqref{gE2} and $\delta$ is constructed in eqn. \eqref{de}.

Let
$$
q\, :\, {\rm Pic}^0(Y_\tau)\, \longrightarrow\,
{\rm Pic}^0(Y_\tau)/\Gamma_r
$$
be the quotient map.
Using the description of the fixed point set for the involution
$\delta_{(1,0)}$ given in Theorem \ref{thm2}, together with the
fact that $r$ is an odd integer, the following description of
the fixed point set for the involution $\delta$ is obtained:
\begin{equation}\label{iia}
({\rm Pic}^0(Y_\tau)/\Gamma_r)^{\delta}\, =\,
\{q(\phi(\underline{t}))\, \mid\, 0\leq t < 1/r\}\bigcup
\{q(\phi(\underline{t+\sqrt{-1}\tau/2}))\,
\mid\, 0\leq t < 1/r\}\, ,
\end{equation}
where $\phi$ is defined in eqn. \eqref{phi}
and $q$ is defined above. From eqn. \eqref{ii}
we know that the fixed point set for the involution
$\delta_{(r,d)}$ is the image
$$
{\mathcal M}_{Y_\tau}(r,d)^{\delta_{(r,d)}} \, =\,
\widetilde{\gamma_W}(({\rm Pic}^0(Y_\tau)/\Gamma_r)^{\delta})\, .
$$

Take any $t\,\in\, {\mathbb R}$ such that
$0\,\leq\, t\, <\, 1/r$. The line bundle $\phi(\underline{t})$
over $Y_\tau$ has the property that there is an isomorphism
$$
\delta_t\, :\, \phi(\underline{t})\, \longrightarrow\,
\sigma^*\overline{\phi(\underline{t})}
$$
such that $(\sigma^*\overline{\delta_t})\circ\delta_t\,=\,
\text{Id}_{\phi(\underline{t})}$; see
Theorem \ref{thm2}. The isomorphism
$$
\beta_t\, :=\, 
\beta\otimes \delta_t\, :\, W\otimes \phi(\underline{t})\,
\longrightarrow\, \sigma^*\overline{W}\otimes
\sigma^*\overline{\phi(\underline{t})}\, =\,
 \sigma^*(\overline{W\otimes \phi(\underline{t})})
$$
satisfies the identity $(\sigma^*\overline{\beta_t})\circ
\beta_t\, =\, \text{Id}_{W\otimes \phi(\underline{t})}$;
recall that $\beta$ is the isomorphism corresponding to
$E\, \in\, {\mathcal M}_X(r,d)$.

Therefore, $\gamma_W(\phi(\underline{t}))\, =\,
W\bigotimes \phi(\underline{t})$ corresponds to a vector
bundle in ${\mathcal M}_X(r,d)$.

Now take $\mu\, =\, t+\sqrt{-1}\tau/2$, where
$0\,\leq\, t\, <\, 1/r$. The line bundle $\phi(\underline{\mu})$
over $Y_\tau$ has the property that there is an isomorphism
$$
\delta'_\mu\, :\, \phi(\underline{\mu})\, \longrightarrow\,
\sigma^*\overline{\phi(\underline{\mu})}
$$
such that $(\sigma^*\overline{\delta'_\mu})\circ\delta'_\mu
\,=\, -\text{Id}_{\phi(\underline{\mu})}$; see
eqn. \eqref{id.r}. Consequently, the isomorphism
$$
\beta'_\mu\, :=\,
\beta\otimes \delta'_\mu\, :\, W\otimes \phi(\underline{\mu})\,
\longrightarrow\, \sigma^*\overline{W\otimes \phi(\underline{\mu})}
$$
satisfies the identity $(\sigma^*\overline{\beta'_\mu})\circ
\beta'_\mu\, =\, -\text{Id}_{W\otimes \phi(\underline{\mu})}$.

Since $W\bigotimes \phi(\underline{\mu})$ is stable, any
isomorphism from $W\bigotimes \phi(\underline{\mu})$ to
$\sigma^*\overline{W\bigotimes \phi(\underline{\mu})}$ must be
of the form $\lambda\beta'_\mu$ for some $\lambda\, \in\,
{\mathbb C}\setminus\{0\}$. As
$$
(\sigma^*\overline{\lambda\beta'_\mu})\circ
(\lambda\beta'_\mu)\, =\, \lambda\overline{\lambda}
\cdot (\sigma^*\overline{\beta'_\mu})\circ
\beta'_\mu\, =\, -\vert\lambda\vert^2\cdot
\text{Id}_{W\otimes \phi(\underline{\mu})}
\, \not=\, 1\, ,
$$
we conclude that there is no $E\, \in\, {\mathcal M}_X(r,d)$
such that $\gamma_W(\phi(\underline{\mu}))\, =\,
W\bigotimes \phi(\underline{\mu})$ is isomorphic
to $E\bigotimes_{\mathbb R}\mathbb C$.
This completes the proof of the theorem.
\end{proof}

\subsection{Classification of stable vector bundles}\label{st62}

We will now consider the general case where the rank and
the degree are not necessarily mutually coprime. The following
theorem of \cite{At2} and \cite{Tu} will be very useful.

\begin{theorem}\label{thm3}
Let $V$ be a stable vector bundle over a smooth
elliptic curve defined over $\mathbb C$. Then
${\rm rank}(V)$ and ${\rm degree}(V)$ are mutually coprime.
\end{theorem}

See \cite[p. 20, Appendix A]{Tu} for a proof of
Theorem \ref{thm3}.

We note that for any Klein bottle $X$,
there are stable vector bundles of rank two
and degree two over $X$ (see Corollary \ref{cor1}), and
also there are stable vector bundles of rank two
and degree zero over $X$ (see Corollary \ref{cor2}).

Let $X$ be a geometrically connected smooth projective curve of
genus one, defined over the real numbers,
which does not have any points defined over $\mathbb R$.
Let $\sigma$ be the
anti--holomorphic involution of $X_{\mathbb C}\, =\,
X\times_{\mathbb R} {\mathbb C}$ as in eqn. \eqref{sig}.

With the above notation, we have the following lemma:

\begin{lemma}\label{st.de.}
Let $V$ be a stable real algebraic vector bundle over $X$. Let
$V_{\mathbb C}\, :=\, V\bigotimes_{\mathbb R}{\mathbb C}$ be the
corresponding complex algebraic vector bundle over $X_{\mathbb C}
\, :=\, X \times_{\mathbb R}{\mathbb C}$.
Then either $V_{\mathbb C}$
is stable, or $V_{\mathbb C}$ is isomorphic to
$F\bigoplus \sigma^*\overline{F}$, where $F$ is a stable
vector bundle over $X_{\mathbb C}$.
\end{lemma}

\begin{proof}
Let
\begin{equation}\label{d.}
\delta\, :\, V_{\mathbb C}\, \longrightarrow\,
\sigma^*\overline{V_{\mathbb C}}
\end{equation}
be the isomorphism as in eqn. \eqref{tau}. The vector bundle
$V_{\mathbb C}$ is polystable, as $V$ is so (see Lemma
\ref{lem1}(3)). Assume that $V_{\mathbb C}$ is not stable. Let
\begin{equation}\label{fi}
V_{\mathbb C}\, =\, \bigoplus_{i=1}^\ell F_i
\end{equation}
be a decomposition of $V_{\mathbb C}$ into a direct sum of
stable vector bundles.

Consider the holomorphic subbundle $\sigma^*\overline{F_1}\, \subset
\,\sigma^*\overline{V_{\mathbb C}}$, where $F_1$ is the subbundle
in eqn. \eqref{fi}. Let
$$
F'\, :=\, \delta^{-1}(\sigma^*\overline{F_1})\, \subset\,
V_{\mathbb C}
$$
be the subbundle, where $\delta$ is the isomorphism in eqn.
\eqref{d.}. Let
$$
\psi\, :\, F'\, \longrightarrow\, V_{\mathbb C}/F_1
$$
be the natural projection.

Note that $V_{\mathbb C}/F_1$ is polystable, $F'$ is stable,
and
$$
\frac{\text{degree}(V_{\mathbb C}/F_1)}{\text{rank}(V_{\mathbb C}
/F_1)}\, =\, \frac{\text{degree}(F')}{\text{rank}(F')}\, .
$$
Therefore, the above homomorphism $\psi$ is either the zero
homomorphism, or it is a fiberwise injective homomorphism of
vector bundles. If $\psi\, =\, 0$, then $F'\, =\, F_1$. In that
case $F_1$ defines a real algebraic subbundle
$F_{\mathbb R}$ of $V$ with
$$
\frac{\text{degree}(F_{\mathbb R})}{\text{rank}(F_{\mathbb R})}
\, =\, \frac{\text{degree}(V)}{\text{rank}(V)}\, .
$$
But this contradicts the assumption that $V$ is stable.

Therefore, $\psi$ is a fiberwise injective homomorphism of
vector bundles. This implies that $F_1+F'$ is a subbundle
of $V_{\mathbb C}$, and $F_1+F'$ is identified with
$F_1\bigoplus F'$; by $F_1+F'$ we denote the coherent
subsheaf of $V_{\mathbb C}$ generated by $F_1$ and $F'$.
Furthermore, the isomorphism $\delta$
takes the subbundle $F_1\bigoplus F'\, \subset\,
V_{\mathbb C}$ to the subbundle
$\sigma^*(\overline{F_1\bigoplus F'})
\, \subset\, \sigma^*\overline{V_{\mathbb C}}$. This
follows from the fact that $(\sigma^*\overline{\delta})\circ
\delta \, =\, \text{Id}_{V_{\mathbb C}}$.
Therefore, $F_1\bigoplus F'$ defines a real algebraic
subbundle $W$ of $V$
with
$$
\frac{\text{degree}(W)}{\text{rank}(W)}
\, =\, \frac{\text{degree}(V)}{\text{rank}(V)}\, .
$$
Since $V$ is stable, from this we conclude that
$F_1\bigoplus F'\, =\, V_{\mathbb C}$. Since
$F'\, =\, \delta^{-1}(\sigma^*\overline{F_1})$ is isomorphic
to $\sigma^*\overline{F_1}$, the proof of the lemma is complete.
\end{proof}

Theorem \ref{thm3} and Lemma \ref{st.de.} have the following
corollary:

\begin{corollary}\label{cor3}
Let $E$ be a stable real algebraic vector bundle
of rank $r$ and degree $d$ over a Klein bottle $X$. Then either
${\rm gcd} (r\, ,d)\, =\,1$ or ${\rm gcd}(r\, ,d)\, =\,2$.
\end{corollary}

\begin{proposition}\label{hro}
Let $X$ be a Klein bottle, and $X_{\mathbb C}\, :=\,
X\times_{\mathbb R} {\mathbb C}$. Let $\sigma$ be
the anti--holomorphic involution of $X_{\mathbb C}$.
Fix a positive integer $r'$ and an odd integer $d'$
such that $r'$ and $d'$ are mutually coprime.
Let $r\, :=\, 2r'$ and $d\, :=\, 2d'$.
\begin{enumerate}
\item For any stable vector bundle $V$ over $X_{\mathbb C}$
of rank $r'$ and degree $d'$, the real
algebraic vector bundle over $X$
defined by $V\bigoplus \sigma^*\overline{V}$ is stable.

\item For any stable real algebraic vector bundle $E$ 
over $X$ of rank $r$ and degree $d$,
there is a stable vector bundle $V$ over $X_{\mathbb C}$
of rank $r'$ and degree $d'$ such that
$$
E\otimes_{\mathbb R}{\mathbb C}\, =\,
V\oplus \sigma^*\overline{V}\, .
$$

\item Let $V$ and $W$ be stable vector bundles
of rank $r'$ and degree $d'$ over
$X_{\mathbb C}$, and let
$V_{\mathbb R}$ (respectively, $W_{\mathbb R}$)
be the real algebraic vector bundle over $X$ given
by $V\bigoplus \sigma^*\overline{V}$ (respectively,
$W\bigoplus \sigma^*\overline{W}$). Then
$V_{\mathbb R}$ and $W_{\mathbb R}$ are isomorphic if
and only if either $V$ is isomorphic to $W$ or $V$
is isomorphic to $\sigma^*\overline{W}$.
\end{enumerate}
\end{proposition}

\begin{proof}
Let $V$ be a stable vector bundle over $X_{\mathbb C}$ of rank
$r'$ and degree $d'$. Since $V\bigoplus \sigma^*\overline{V}$
is polystable, the real algebraic vector bundle $E$ over $X$
corresponding to $V\bigoplus \sigma^*\overline{V}$ is also
polystable (see Lemma \ref{lem1}(3)). Assume that $E$ is not
stable. Therefore, there is nonzero proper subbundle
$E'\, \subset\, E$ with
\begin{equation}\label{pi}
\frac{\text{degree}(E')}{\text{rank}(E')}
\, =\, \frac{d'}{r'}\, .
\end{equation}

Using eqn. \eqref{pi} together with the conditions on $r'$ and
$d'$  it follows that $\text{degree}(E')\, =\, d'$. Since there
are no real algebraic vector bundles over $X$ of odd degree
(Proposition \ref{prop1}), and $d'$ is an odd integer, we get
a contradiction. Therefore, the vector bundle $E$ must be stable.

To prove the second statement, for any stable real algebraic
vector bundle $E$ over $X$ of rank $r$ and
degree $d$, consider the complex algebraic
vector bundle $E_{\mathbb C}\, :=\, E\bigotimes_{\mathbb
R}{\mathbb C}$ over $X_{\mathbb C}$. From Lemma \ref{lem1}(3)
it follows that $E_{\mathbb C}$ is polystable, and from
Theorem \ref{thm3} it follows that $E_{\mathbb C}$ is
not stable. Therefore, by Lemma \ref{st.de.},
the vector bundle $E_{\mathbb C}$ is
of the form $F\bigoplus \sigma^*\overline{F}$, where $F$
is a stable vector bundle over $X_{\mathbb C}$ of rank $r'$
and degree $d'$.

The third statement follows from \cite[p. 315, Theorem 2]{At1}
and Lemma \ref{lem0}. This completes the proof of the proposition.
\end{proof}

Let ${\mathcal M}_{X_{\mathbb C}}(r',d')$ be the moduli space of
stable vector bundles over $X_{\mathbb C}$ of rank $r'$ and degree
$d'$, where $r'$ and $d'$ are as in Proposition \ref{hro}.
As in eqn. \eqref{gar}, let
$$
\Gamma_r\, \subset\, \text{Pic}^0(X_{\mathbb C})
$$
be the subgroup of line bundles $L$ over $X_{\mathbb C}$ such that
$L^{\otimes r}\, \cong\, {\mathcal O}_{X_{\mathbb C}}$. If we
fix any $E\, \in\, {\mathcal M}_{X_{\mathbb C}}(r',d')$, then the
map
$$
\text{Pic}^0(X_{\mathbb C})\, \longrightarrow\,
{\mathcal M}_{X_{\mathbb C}}(r',d')
$$
defined by
$$
L\, \longmapsto\, E\otimes L
$$
induced an algebraic
isomorphism of $\text{Pic}^0(X_{\mathbb C})/\Gamma_r$ with
${\mathcal M}_{X_{\mathbb C}}(r',d')$ (see Theorem \ref{thm.at2}).
If we fix $E\, =\, V\bigotimes_{\mathbb R}{\mathbb C}$, where
$V$ is a real algebraic vector bundle over $X$ (such vector
bundles exist by Theorem \ref{prop5}), then the involution
of ${\mathcal M}_{X_{\mathbb C}}(r',d')$ defined by
$W\, \longrightarrow\, \sigma^*\overline{W}$ corresponds to the
involution of $\text{Pic}^0(X_{\mathbb C})$ defined by
$L\, \longrightarrow\, \sigma^*\overline{L}$; see
eqn. \eqref{ii}.
Therefore, Proposition \ref{hro} has the following corollary:

\begin{corollary}\label{cor4}
Take $r\, :=\, 2r'$ and $d\, :=\, 2d'$, where $d'$ is an
odd integer and $r'$ is a positive integer coprime to $d'$.
The set of isomorphism classes of stable
real algebraic vector bundles over $X$ of
rank $r$ and degree $d$ is canonically identified
with the quotient space
${\mathcal M}_{X_{\mathbb C}}(r',d')/({\mathbb Z}/2{\mathbb Z})$
for the involution
of ${\mathcal M}_{X_{\mathbb C}}(r',d')$ defined by
$W\, \longrightarrow\, \sigma^*\overline{W}$. They are also
isomorphic to the quotient space
$({\rm Pic}^0(X_{\mathbb C})/\Gamma_r)/({\mathbb Z}/2{\mathbb Z})$
for the involution of ${\rm Pic}^0(X_{\mathbb C})/\Gamma_r$ defined
by $L\, \longmapsto\, \sigma^*\overline{L}$.
\end{corollary}

We will now consider the case where $d'$ is an even integer.

\begin{proposition}\label{hro2}
Fix a positive integer $r'$ and an even integer $d'$
such that $r'$ and $d'$ are mutually coprime.
Let ${\mathcal M}_{X_{\mathbb C}}(r',d')
\setminus {\mathcal M}_{X}(r',d')$ denote the
set of isomorphism classes of stable vector bundles over
$X_{\mathbb C}$ of rank $r'$ and degree $d'$ which are not
of the form $F\bigotimes_{\mathbb R} {\mathbb C}$, where
$F$ is some real algebraic vector bundle over $X$.
Set $r\, :=\, 2r'$ and $d\, :=\, 2d'$.

\begin{enumerate}
\item For any stable vector bundle $E\, \in\,
{\mathcal M}_{X_{\mathbb C}}(r',d')
\setminus {\mathcal M}_{X}(r',d')$
over $X_{\mathbb C}$, the real algebraic vector bundle over $X$
defined by $E\bigoplus \sigma^*\overline{E}$ is stable.

\item For any stable real algebraic vector bundle $V$
over $X$ of rank $r$ and degree $d$,
there is a stable vector bundle $E\, \in\,
{\mathcal M}_{X_{\mathbb C}}(r',d')
\setminus {\mathcal M}_{X}(r',d')$
over $X_{\mathbb C}$ such that
$$
V\otimes_{\mathbb R}{\mathbb C}\, =\,
E\oplus \sigma^*\overline{E}\, .
$$

\item Take vector bundles $V\, ,W\, \in\,
{\mathcal M}_{X_{\mathbb C}}(r',d')
\setminus {\mathcal M}_{X}(r',d')$. Let
$V_{\mathbb R}$ (respectively, $W_{\mathbb R}$)
be the real algebraic vector bundle over $X$ given
by $V\bigoplus \sigma^*\overline{V}$ (respectively,
$W\bigoplus \sigma^*\overline{W}$). Then
$V_{\mathbb R}$ and $W_{\mathbb R}$ are isomorphic if
and only if either $V$ is isomorphic to $W$ or $V$
is isomorphic to $\sigma^*\overline{W}$.
\end{enumerate}
\end{proposition}

\begin{proof}
Take any $E\, \in\, {\mathcal M}_{X_{\mathbb C}}(r',d')
\setminus {\mathcal M}_{X}(r',d')$.
Since $E\bigoplus \sigma^*\overline{E}$ is polystable,
by Lemma \ref{lem1}(3), the real algebraic
vector bundle $V$ over $X$ defined by $E\bigoplus
\sigma^*\overline{E}$ is polystable. Assume that
$V$ is not stable. So $V\, =\, V_1\bigoplus V_2$, where
both $V_1$ and $V_2$ are nonzero real
algebraic vector bundles. Since
$$
E\oplus \sigma^*\overline{E}\, =\, V\otimes_{\mathbb R}
{\mathbb C} \, =\, (V_1\otimes_{\mathbb R}
{\mathbb C}) \oplus (V_2\otimes_{\mathbb R}
{\mathbb C})\, ,
$$
and both $E$ and $\sigma^*\overline{E}$ are
indecomposable, from
\cite[p. 315, Theorem 2]{At1} we conclude that either
$E\, =\, V_1\bigotimes_{\mathbb R}
{\mathbb C}$ or $E\, =\, V_2\bigotimes_{\mathbb R}
{\mathbb C}$. This contradicts the assumption that
$E\, \in\, {\mathcal M}_{X_{\mathbb C}}(r',d')
\setminus {\mathcal M}_{X}(r',d')$. Therefore, the real
algebraic vector bundle $V$ is stable.

To prove the second statement, let $V$ be a
stable real algebraic vector bundle
over $X$ of rank $r$ and degree $d$. Then the complex vector
bundle
$$
V_{\mathbb C}\, :=\, V\otimes_{\mathbb R} {\mathbb C}
$$
over $X_{\mathbb C}$ is polystable (see
Lemma \ref{lem1}(3)).

Since $\text{gcd}(r,d)\, \not=\, 1$, the
polystable vector bundle $V_{\mathbb C}$ is not stable (see
Theorem \ref{thm3}). Since $\text{gcd}(r,d)\, =\, 2$,
we have
$$
V_{\mathbb C}\, =\, E_1\oplus E_2\, ,
$$
where both $E_1$ and $E_2$ are nonzero stable vector bundles.

Let $\delta\, :\, V_{\mathbb C}\, \longrightarrow\,
\sigma^*\overline{V_{\mathbb C}}$
be the isomorphism as in eqn. \eqref{tau}. Since $V$
is stable, $\delta(E_1)\, \not=\, \sigma^*\overline{E_1}$.
Therefore, the composition homomorphism
$$
E_1\, \stackrel{\delta}{\longrightarrow}\,
\sigma^*\overline{V_{\mathbb C}}
\, \longrightarrow\, (\sigma^*\overline{V_{\mathbb
C}})/\sigma^*\overline{E_1}\, =\, \sigma^*\overline{E_2}
$$
is nonzero. Since $E_1$ and $\sigma^*\overline{E_2}$ are
stable with
$$
\frac{\text{degree}(E_1)}{\text{rank}(E_1)} \, =\,
\frac{\text{degree}(\sigma^*
\overline{E_2})}{\text{rank}(\sigma^*\overline{E_2})}\, ,
$$
this implies that the above composition
homomorphism $E_1\, \longrightarrow\,\sigma^*\overline{E_2}$
is an isomorphism.

To complete the proof of the second statement we need
to show that $E_1\, \not=\, F\bigotimes_{\mathbb R}{\mathbb C}$
for some real algebraic vector bundle $F$ over $X$. 
Assume that $E_1\, =\, F\bigotimes_{\mathbb R}{\mathbb C}$.
Then 
$$
V_{\mathbb C}\, =\, E_1\oplus \sigma^*\overline{E_1}\, =\,
(F\oplus F)\otimes_{\mathbb R}{\mathbb C}\, .
$$
Now from Lemma \ref{lem0} it follows that
$V\, =\, F\oplus F$. This contradicts
the assumption that $V$ is stable. This completes
the proof of the second statement

The third statement follows from 
\cite[p. 315, Theorem 2]{At1} and Lemma \ref{lem0}. This
completes the proof of the proposition.
\end{proof}

The following corollary is deduced using Proposition
\ref{hro2} just as Corollary \ref{cor4} is
deduced from Proposition \ref{hro}.

\begin{corollary}\label{cor5}
Take $r'$ and $d'$ as in Proposition \ref{hro2}. Let
${\mathcal M}_{X_{\mathbb C}}(r',d')$ be the moduli space of
stable vector bundles over $X_{\mathbb C}$ of rank $r'$ and degree
$d'$. Let ${\mathcal M}_{X_{\mathbb C}}(r',d')
\setminus {\mathcal M}_{X}(r',d')\, \subset\,
{\mathcal M}_{X_{\mathbb C}}(r',d')$ denote the subset
defined by all stable vector bundles which are not of the
form $F\bigotimes_{\mathbb R}{\mathbb C}$, where
$F$ is some real algebraic vector bundle over $X$.
Then the set of all isomorphism classes of stable
real algebraic vector bundles over $X$ of
rank $r$ and degree $d$ is canonically identified
with the quotient space $({\mathcal M}_{X_{\mathbb C}}(r',d')
\setminus {\mathcal M}_{X}(r',d'))/({\mathbb Z}/2{\mathbb Z})$
for the involution of ${\mathcal M}_{X_{\mathbb C}}(r',d')
\setminus {\mathcal M}_{X}(r',d')$ defined by
$W\, \longrightarrow\, \sigma^*\overline{W}$.
\end{corollary}

Let ${\rm Pic}^0(X)\, \subset\, {\rm Pic}^0(X_{\mathbb C})$
be the subset consisting of all line bundles over
$X_{\mathbb C}$ of the form
$L\bigotimes_{\mathbb R}{\mathbb C}$, where $L$ is some
real algebraic line bundle over $X$.
The space $({\mathcal M}_{X_{\mathbb C}}(r',d')
\setminus {\mathcal M}_{X}(r',d'))/({\mathbb Z}/2{\mathbb 
Z})$ in Corollary \ref{cor5} is
isomorphic to the quotient space
$(({\rm Pic}^0(X_{\mathbb C})/\Gamma_r)\setminus
({\rm Pic}^0(X)/\Gamma^{\mathbb R}_r))/({\mathbb Z}/2{\mathbb
Z})$
for the involution of ${\rm Pic}^0(X_{\mathbb C})/\Gamma_r$ defined
by $L\, \longmapsto\, \sigma^*\overline{L}$, where $\Gamma_r$
is the group of line bundles $L$ over $X_{\mathbb C}$
with $L^{\otimes r}\, \cong\, {\mathcal O}_{X_{\mathbb C}}$, and
$\Gamma^{\mathbb R}_r\, :=\, {\rm Pic}^0(X)\bigcap \Gamma_r$
as in eqn. \eqref{garr}.

\begin{remark}
{\rm In view of Corollary \ref{cor3}, we conclude that
Theorem \ref{prop5}, Proposition \ref{hro} and
Proposition \ref{hro2} together classify all stable real
algebraic vector bundles over $X$. Therefore, we have
classified all polystable real
algebraic vector bundles over $X$.}
\end{remark}

\medskip
\noindent
\textsc{Acknowledgements.} We thank the referee for helpful
comments to improve the exposition.


\end{document}